\documentclass[11pt,leqno]{article}
\usepackage{amsfonts}
\usepackage{amsmath, amsthm, amsfonts, amssymb, color}
\usepackage{mathrsfs}
\renewcommand{\bar}{\overline}
\renewcommand{\hat}{\widehat}
\renewcommand{\tilde}{\widetilde}

\newtheorem{thm}{Theorem}[section]

\newtheorem{lem}[thm]{Lemma}

\newtheorem{exa}[thm]{Example}
\theoremstyle{definition}

\newcommand{\scr}[1]{\mathscr #1}
\definecolor{wco}{rgb}{0.5,0.2,0.3}

\numberwithin{equation}{section} \theoremstyle{remark}
\newtheorem{rem}{Remark}[section]

\newcommand{\ua}{\uparrow}

\title{{\bf Least squares estimator for path-dependent McKean-Vlasov SDEs via discrete-time observations}
}
\author{
{\bf  Panpan Ren and Jiang-Lun Wu}\\
\footnotesize{Department of Mathematics, Swansea University,
Singleton Park, SA2 8PP, UK}\\ 
\footnotesize{  673788@swansea.ac.uk, J.L.Wu@swansea.ac.uk}}

\begin{document}
\def\R{\mathbb R}  \def\ff{\frac} \def\ss{\sqrt} \def\B{\mathbf
B}
\def\N{\mathbb N} \def\kk{\kappa} \def\m{{\bf m}}
\def\dd{\delta} \def\DD{\Delta} \def\vv{\varepsilon} \def\rr{\rho}
\def\<{\langle} \def\>{\rangle} \def\GG{\Gamma} \def\gg{\gamma}
  \def\nn{\nabla} \def\pp{\partial} \def\EE{\scr E}
\def\d{\text{\rm{d}}} \def\bb{\beta} \def\aa{\alpha} \def\D{\scr D}
  \def\si{\sigma} \def\ess{\text{\rm{ess}}}
\def\beg{\begin} \def\beq{\begin{equation}}  \def\F{\scr F}
\def\Ric{\text{\rm{Ric}}} \def\Hess{\text{\rm{Hess}}}
\def\e{\text{\rm{e}}} \def\ua{\underline a} \def\OO{\Omega}  \def\oo{\omega}
 \def\tt{\tilde} \def\Ric{\text{\rm{Ric}}}
\def\cut{\text{\rm{cut}}} \def\P{\mathbb P} \def\ifn{I_n(f^{\bigotimes n})}
\def\C{\scr C}      \def\aaa{\mathbf{r}}     \def\r{r}
\def\gap{\text{\rm{gap}}} \def\prr{\pi_{{\bf m},\varrho}}  \def\r{\mathbf r}
\def\Z{\mathbb Z} \def\vrr{\varrho} \def\ll{\lambda}
\def\L{\scr L}\def\Tt{\tt} \def\TT{\tt}\def\II{\mathbb I}
\def\i{{\rm in}}\def\Sect{{\rm Sect}}\def\E{\mathbb E} \def\H{\mathbb H}
\def\M{\scr M}\def\Q{\mathbb Q} \def\texto{\text{o}} \def\LL{\Lambda}
\def\Rank{{\rm Rank}} \def\B{\scr B} \def\i{{\rm i}} \def\HR{\hat{\R}^d}
\def\to{\rightarrow}\def\l{\ell}
\def\8{\infty}\def\X{\mathbb{X}}\def\3{\triangle}
\def\V{\mathbb{V}}\def\M{\mathbb{M}}\def\W{\mathbb{W}}\def\Y{\mathbb{Y}}\def\1{\lesssim}

\def\La{\Lambda}\def\S{\mathbf{S}}

\renewcommand{\bar}{\overline}
\renewcommand{\hat}{\widehat}
\renewcommand{\tilde}{\widetilde}

\maketitle

\begin{abstract}
In this paper, we are  interested in least squares estimator for a
class of path-dependent McKean-Vlasov stochastic differential
equations (SDEs). More precisely, we investigate the consistency and
asymptotic distribution of the least squares estimator for the unknown
parameters involved by establishing an appropriate contrast function.
Comparing to the existing results in the literature, the innovations
of our paper lie in three aspects: (i) We adopt a tamed Euler-Maruyama
algorithm to establish the contrast function under the monotone condition,
under which the Euler-Maruyama scheme no longer works; (ii) We take the
advantage of linear interpolation with respect to the discrete-time
observations to approximate the functional solution; (iii) Our model
is more applicable and practice as we are dealing with SDEs with irregular
coefficients (e.g.,  H\"older continuous) and path-distribution dependent.

\end{abstract}

\section{Introduction and main results}
We start with some notation and terminology. Let
$(\R^d,\<\cdot,\cdot\>,|\cdot|)$ be the $d$-dimensional Euclidean
space, and $\R^d\otimes\R^m$ the collection of all $d\times m$
matrices endowed with the Hilbert-Schmidt norm $\|\cdot\|$. For
fixed $r_0>0$, $\C:=C([-r_0,0];\R^d)$ stands for the family of all
continuous functions $f:[-r_0,0]\rightarrow\R^d$ which is a Banach space
with the uniform norm $\|f\|_\8:=\sup_{-r_0\le v\le 0}|f(v)|$. Given any
integer $ p\ge1$, we use $\Theta$ to denote a bounded, open and convex
subset of $\R^p$ whose closure is written as $\bar\Theta$. Let $\mathcal
{P}(\C)$ be the totality of all probability measures on $\C$. Set
$\mathcal {P}_2(\C):=\{\mu\in\mathcal
{P}(\C):\mu(\|\cdot\|_\8^2):=\int_\C\|\xi\|_\8^2\mu(\d\xi)<\8\}$.
$(\mathcal {P}_2(\C),\W_2)$ is a Polish space under the Warsserstein
distance $\mathbb{W}_2$ on $\mathcal {P}_2(\C)$ defined by
$$\W_2(\mu,\nu):= \inf_{\pi\in \C(\mu,\nu)} \bigg(\int_{\C\times\C} \|\xi-\eta\|^2_\8\pi(\d \xi,\d \eta)\bigg)^{\ff 1 2},\ \ \mu,\nu\in \mathcal
{P}_2(\C),$$  where $\C(\mu,\nu)$ is the set of couplings for $\mu$
and $\nu$. As usual, we use $\lfloor a\rfloor$ to denote the integer part of $a\ge0.$

The time evolution for most of stochastic dynamical systems depends
not only on the present state but also on the past path. So,
path-dependent (i.e., functional) SDEs are much more desirable; see,
e.g., the monograph \cite{M84}. Since the pioneer work \cite{IN} due
to It\^o and Nisio, path-dependent SDEs have been investigated
considerably owing to their theoretical and practical importance;
see, e.g., Hairer et al. \cite{HMS11}, Wang \cite{W18} and the
references within.

 McKean-Vlasov SDEs, which are SDEs with coefficients dependent on the
 law, were initiated by \cite{Mc} inspired by Kac's programme in Kinetic
 theory. An excellent and thorough account of the general theory of McKean-Vlasov
 SDEs and their particle approximations can be found in \cite{sz}. McKean-Vlasov
 SDEs
are alternatively referred to as mean-field SDEs in the literature, which have
wide applications in interacting particle systems, optimal control problems,
differential games, just to mention but a few. Recently, McKean-Vlasov SDEs have been
extensively investigated on, e.g., wellposedness of strong/weak solutions (cf.
\cite{DST,HSS,LMb,MV,W16}), Freidlin-Wentzell large deviation principles (cf. \cite{DST}), ergodicity (cf. \cite{B14,EGZ,Ve}), links with nonlinear partial differential
equations (cf. \cite{BLPR,HRW,HRW}), and distribution properties (cf.\cite{Huang,W18}).

On the other hand, from stochastic and/or statistical aspects, there exist
unknown parameters in various type SDEs arising in mathematical modeling (cf.
\cite{B08}). Hence, there are vast of investigations paying attention to
parameter estimations for SDEs via maximum likelihood estimator,
least squares estimator (LSE for short), trajectory-fitting estimator, among others.
See, for instance, \cite{Ku04,LS01,M05,P199,SY06}. In the same vein, the parameter estimations for SDEs (without path-dependence) with small noises have been developed very
well; see, e.g., \cite{GS,HL,Long10,L09,LMS,LSS,M10,SM03,U04,U08}, and references
therein.

From above discussion, it is very natural to consider SDEs together with all four features of path dependence, distribution dependence, small noises and unknown parameter. So,
in the present work, we focus on the following path-distribution SDE
\begin{equation}\label{eq1}
\d X^\vv(t)=b(X_t^\vv,\mathscr{L}_{X_t^\vv},\theta)\d
t+\vv\,\si(X_t^\vv,\mathscr{L}_{X_t^\vv})\d B(t),
~~~t>0,~~~~X_0^\vv=\xi\in\C.
\end{equation}
Herein, $\vv\in(0,1)$ is the scale parameter; for fixed $t$,
$X_t^\vv(v):=X^\vv(t+v), v\in[-r_0,0],$ is called the segment (or
window) process generated by $X^\vv(t)$; $\mathscr{L}_{X_t^\vv}$
stands for the distribution of $X_t^\vv$; $b:\C\times\mathcal
{P}_2(\C)\times\Theta\rightarrow\R^d$ and $\si:\C\times\mathcal
{P}_2(\C)\rightarrow\R^d\otimes\R^m$ are continuous w.r.t.  the
first variable and the second variable;
  $\Theta\ni\theta $ is an unknown parameter whose true value is
  written as
$\theta_0\in\Theta$; and $(B(t))_{t\ge0}$ is an $m$-dimensional
Brownian motion on a filtered probability space
$(\OO,\F,(\F_t)_{t\ge0},\P)$ satisfying the usual conditions, that
is, $\F_t$ is non-decreasing (i.e., $\F_s\subseteq\F_t, s\le t),$
$\F_0$ contains all $\P$-null sets and $\F_t$ is right continuous
(i.e., $\F_t=\F_{t+}:=\bigcap_{s\uparrow t}\F_s$).

To guarantee the existence and uniqueness of solutions to \eqref{eq1},
we assume that, for any $\zeta_1,\zeta_2\in\C$,  $\mu,\nu\in\mathcal
{P}_2(\C)$, and $\theta\in\Theta,$
\begin{enumerate}
\item[({\bf A1})]  There exist
$\aa_1,\aa_2>0$ such that
\begin{equation*}
\<\zeta_1(0)-\zeta_2(0),b(\zeta_1,\mu,\theta)-b(\zeta_2,\nu,\theta)\>\le\aa_1\|\zeta_1-\zeta_2\|_\8^2+\aa_2\mathbb{W}_2(\mu,\nu)^2;
\end{equation*}

\item[({\bf A2})]  There exist $\bb_1,\bb_2>0$ such that
\begin{equation*}
\|\si(\zeta_1,\mu)-\si(\zeta_2,\nu)\|^2\le
\bb_1\|\zeta_1-\zeta_2\|_\8^2+\bb_2\mathbb{W}_2(\mu,\nu)^2.
\end{equation*}
\end{enumerate}

From \cite[Theorem 3.1]{HRW}, \eqref{eq1} has a unique
strong solution $(X^\vv(t))_{t\ge-r_0}$ under the assumptions ({\bf
A1}) and ({\bf A2}). For any $\zeta_1,\zeta_2\in\C$,
$\mu,\nu\in\mathcal {P}_2(\C)$, and $\theta\in\Theta,$ if there
exist $\aa,\bb>0$ such that
\begin{equation*}
\<\zeta_1(0)-\zeta_2(0),b(\zeta_1,\mu,\theta)-b(\zeta_2,\mu,\theta)\>\le\aa\|\zeta_1-\zeta_2\|_\8^2
\end{equation*}
and
\begin{equation*}
|b(\zeta_2,\mu,\theta)-b(\zeta_2,\nu,\theta)|\le
\bb\mathbb{W}_2(\mu,\nu),
\end{equation*}
then  ({\bf A1}) holds.

Without loss of generality, we arbitrarily fix the time horizontal $T>0$ and assume that there exist positive integers $n,M$ sufficiently large such that
$\dd:=\frac{T}{n}=\ff{r_0}{M}$. Now we define the continuous-time
tamed Euler-Maruyama (EM) scheme (see, e.g., \cite{HJK}) associated
with \eqref{eq1}
\begin{equation}\label{q3}
\d Y^\vv(t)=b^{(\dd)}(\bar Y^\vv_{t_\dd},\mathscr{L}_{\bar
Y^\vv_{t_\dd}},\theta)\d t+ \vv\,\si(\bar
Y^\vv_{t_\dd},\mathscr{L}_{\bar Y^\vv_{t_\dd}})\d B(t),~~~~t>0
\end{equation}
with the initial value $Y^\vv(t)=X^\vv(t)=\xi(t)$ for any $
t\in[-r_0,0]$, where
\begin{itemize}
\item $t_\dd:=\lfloor t/\dd\rfloor\dd$ for $t\ge0;$

\item For  any $\zeta\in\C$ and $\mu\in\mathcal {P}_2(\C)$,
\begin{equation}\label{e4}
b^{(\dd)}(\zeta,\mu,\theta):=\ff{b(\zeta,\mu,\theta)}{1+\dd^\aa|b(\zeta,\mu,\theta)|},~~~~\aa\in(0,1/2];
\end{equation}

\item For $k=0,1,\cdots,n,$
$\bar Y_{k\dd}^\vv=\{ \bar Y_{k\dd}^\vv(s):-r_0\le s\le0\}$, a
$\C$-valued random variable, is  defined by
\begin{equation}\label{w2}
\bar
Y_{k\dd}^\vv(s)=Y^\vv((k-i)\dd)+\ff{s+i\dd}{\dd}\{Y^\vv((k-i)\dd)
-Y^\vv((k-i-1)\dd)\}
\end{equation}
for any $s\in[-(i+1)\dd,-i\dd]$, $i=0,1,\cdots,M-1$, that is, $\bar
Y_{k\dd}^\vv$ is the linear interpolation of the points
$(Y^\vv(l\dd))_{l=k-M,\cdots,k}$.
\end{itemize}

We denote $(Y_t^\vv)_{t\ge0}$ by the segment process generated by
$(Y^\vv(t))_{t\ge-r_0}$. It is worthy to point out that $\bar
Y^\vv_{t_\dd}\in\C$ is defined by \eqref{w2} rather than by $\bar
Y^\vv_{t_\dd}(s)=\bar Y^\vv(t_\dd+s)$ for any $s\in[-r_0,0]$.% So the
%definitions of $Y_t^\vv$ and $\bar Y^\vv_{t_\dd}$ are quite
%different although they have the similar style.
 Based on the
continuous-time tamed EM algorithm \eqref{q3}, we design the
following contrast function
\begin{equation}\label{eq2}
\Psi_{n,\vv}(\theta)=\vv^{-2}\delta^{-1}\sum_{k=1}^nP_k^*(\theta)\hat\si(\bar
Y_{(k-1)\dd}^\vv)P_k(\theta),
\end{equation}
in which, for  $k=1,\cdots,n$,
\begin{equation}\label{w1}
P_k(\theta):=Y^\vv(k\dd)-Y^\vv((k-1)\dd)-b^{(\dd)}(\bar
Y_{(k-1)\dd}^\vv,\mathscr{L}_{\bar Y_{(k-1)\dd}^\vv},\theta)\dd,~
\hat\si(\bar Y_{k\dd}^\vv):=(\si\si^*)^{-1}(\bar
Y_{k\dd}^\vv,\mathscr{L}_{\bar Y_{k\dd}^\vv}).
\end{equation}
For more motivations on the construction of constrast function
above, we refer to Ren-Wu \cite{RW}. To obtain the LSE of
$\theta\in\Theta$, it   is sufficient  to choose an element
$\hat\theta_{n,\vv}\in\Theta$ satisfying
\begin{equation*}
\Psi_{n,\vv}(\hat\theta_{n,\vv})=\min_{\theta\in\Theta}\Psi_{n,\vv}(\theta).
\end{equation*}
Whence, for
\begin{equation*}
\Phi_{n,\vv}(\theta):=\vv^2(\Psi_{n,\vv}(\theta)-\Psi_{n,\vv}(\theta_0)),
\end{equation*}
one has
\begin{equation}\label{eq4}
\Phi_{n,\vv}(\hat\theta_{n,\vv})=\min_{\theta\in\Theta}\Phi_{n,\vv}(\theta).
\end{equation}
We shall rewrite $\hat\theta_{n,\vv}\in\Theta$ such that \eqref{eq4}
holds true as
\begin{equation*}
\hat\theta_{n,\vv}=\arg\min_{\theta\in\Theta}\Phi_{n,\vv}(\theta),
\end{equation*}
which   is called the LSE of the unknown parameter
$\theta\in\Theta$.

To discuss the consistency of LSE (see Theorem \ref{th1} below), we
further suppose that, for any $\zeta,\zeta_1,\zeta_2\in\C$,
$\mu,\nu\in\mathcal {P}_2(\C)$, and $\theta\in\Theta,$
\begin{enumerate}

\item[({\bf B1})] There exist $q_1,L_1>0$ such that
\begin{equation*}
|b(\zeta_1,\mu,\theta)-b(\zeta_2,\nu,\theta)|\le
L_1\Big\{(1+\|\zeta_1\|_\8^{q_1}+\|\zeta_2\|_\8^{q_1})\|\zeta_1-\zeta_2\|_\8+\mathbb{W}_2(\mu,\nu)\Big\};
\end{equation*}
\item[({\bf B2})] There exist $q_2,L_2>0$ such that
\begin{equation*}
\sup_{\theta\in\bar\Theta}\|(\nn_\theta
b)(\zeta_1,\mu,\theta)-(\nn_\theta b)(\zeta_2,\nu,\theta)\|\le
L_2\Big\{(1+\|\zeta_1\|_\8^{q_2}+\|\zeta_2\|_\8^{q_2})\|\zeta_1-\zeta_2\|_\8+\mathbb{W}_2(\mu,\nu)\Big\},
\end{equation*}
where $(\nn_\theta b)$ is the gradient operator w.r.t. the third
spatial variable;

\item[({\bf B3})] $(\si\si^*)(\zeta,\mu)$ is invertible, and
there exist $q_3,L_3>0$ such that
\begin{equation*}
\|(\si\si^*)^{-1}(\zeta_1,\mu)-(\si\si^*)^{-1}(\zeta_2,\nu)\|\le
L_3\Big\{(1+\|\zeta_1\|_\8^{q_3}+\|\zeta_2\|_\8^{q_3})\|\zeta_1-\zeta_2\|_\8+\mathbb{W}_2(\mu,\nu)\Big\};
\end{equation*}
\item[({\bf B4})] There exists a constant $K>0$ such that
\begin{equation*}
|\xi(t)-\xi(s)|\le K|t-s|,~~~t,s\in[-r_0,0],
\end{equation*}
where $\xi(\cdot)$ stands for the initial value of \eqref{eq1}.
\end{enumerate}

In order to reveal the asymptotic distribution of LSE (see Theorem
\ref{th2} below), we   in addition assume that
\begin{enumerate}
\item[({\bf C})] There exist $q_4,L_4>0$ such that
\begin{equation*}
\begin{split}
&\sup_{\theta\in\bar\Theta}\|(\nn_\theta(\nn_\theta
b^*))(\zeta_1,\mu,\theta)-(\nn_\theta(\nn_\theta
b^*))(\zeta_2,\nu,\theta)\|\\
&\le
L_4\Big\{(1+\|\zeta_1\|_\8^{q_4}+\|\zeta_2\|_\8^{q_4})\|\zeta_1-\zeta_2\|_\8+\mathbb{W}_2(\mu,\nu)\Big\},
\end{split}
\end{equation*}
where $b^*$ means the transpose of $b.$
\end{enumerate}

Next we consider the following deterministic path-dependent ordinary
equation
\begin{equation}\label{k1}
\d X^0(t)=b(X_t^0,\mathscr{L}_{X_t^0},\theta_0)\d
t,~~~t>0,~~~X_0^0=\xi\in\C.
\end{equation}
Under the assumption ({\bf A1}), \eqref{k1} is wellposed. In
\eqref{k1}, $\mathscr{L}_{X_t^0}$ is indeed a Dirac's delta measure
at the point $X_t^0$ as $X_t^0$ is deterministic. To unify the
notation, we keep the notation $\mathscr{L}_{X_t^0}$ in lieu of
$\dd_{X_t^0}$. We remark that ({\bf B4}) is imposed to guarantee
that the linear interpolation $\bar Y^\vv_{t_\dd}$ tends to $X_t^0$
in the moment sense, see Lemma \ref{le1} below.

 For any random variable $\zeta\in\C$ with
$\mathscr{L}_\zeta\in\mathcal {P}_2(\C)$,  set
\begin{equation}\label{p1}
\Gamma(\zeta,\theta,\theta_0):=b(\zeta,\mathscr{L}_\zeta,\theta_0)
-b(\zeta,\mathscr{L}_\zeta,\theta),~~
~~\Gamma^{(\dd)}(\zeta,\theta,\theta_0):=b^{(\dd)}(\zeta,\mathscr{L}_\zeta,\theta_0)
-b^{(\dd)}(\zeta,\mathscr{L}_\zeta,\theta),
\end{equation}
and, for any $\theta\in\Theta,$
\begin{equation*}\label{e0}
\Xi(\theta)=\int_0^T\Gamma^*(X_t^0,\theta,\theta_0)\hat\si(X_t^0)\Gamma(X_t^0,\theta,\theta_0)\d
t,
\end{equation*}
where $(X_t^0)_{t\ge0}$ is the functional solution  to \eqref{k1}.

The  theorem  below is concerned with the consistency of the LSE for
the parameter $\theta\in\Theta$, which is the first contribution of
our work.

\begin{thm}\label{th1}
  Let $({\bf A1})-({\bf A2})$ and $({\bf B1})-({\bf B4})$ hold   and
assume further that $\Xi(\theta)>0$ for $\theta\neq\theta_0$. Then
\begin{equation*}
\hat\theta_{n,\vv}\rightarrow\theta_0~~~~\mbox{ in probability as }
 \vv\rightarrow0 ~~\mbox{ and }~~n\to\8.
\end{equation*}

\end{thm}

For $A:=(A_1,A_2,\cdots,A_p)\in\R^p\otimes\R^{pd}$ with  $A_{k}\in
\R^p\otimes\R^d$, $k=1,\cdots,p,$ and $B\in\R^d$, define $A\circ
B\in\R^p\otimes\R^p$ by
\begin{equation*}
A\circ B=(A_1B,A_2B,\cdots,A_pB).
\end{equation*}
For any $\theta\in\Theta$, set
\begin{equation}\label{0z3}
I(\theta):=\int_0^T(\nn_\theta
b)^*(X_t^0,\mathscr{L}_{X_t^0},\theta)\hat\si(X_t^0)(\nn_\theta
b)(X_t^0,\mathscr{L}_{X_t^0},\theta)\d t,
\end{equation}
\begin{equation}\label{0z2}
\begin{split}
K(\theta):&=-2\int_0^T(\nn_\theta^{(2)}
b^*)(X_t^0,\mathscr{L}_{X_t^0},\theta)\circ\Big(\hat\si(X_t^0)
\Gamma(X_t^0,\theta,\theta_0)\Big)\d t,
\end{split}
\end{equation}
where $(\nn_\theta^{(2)} b^*):=(\nn_\theta(\nn_\theta b^*))$, and,
for any random variable $\zeta\in\C$ with
$\mathscr{L}_\zeta\in\mathcal {P}_2(\C)$,
\begin{equation}\label{0s0}
\Upsilon(\zeta,\theta_0)=(\nn_\theta
b)^*(\zeta,\mathscr{L}_{\zeta},\theta_0)\hat\si(\zeta)\si(\zeta,\mathscr{L}_{\zeta}).
\end{equation}

Another main result in this paper is presented as below, which
reveals the asymptotic distribution of $\hat\theta_{n,\vv}.$

\begin{thm}\label{th2}
  Let the assumptions of Theorem \ref{th1} hold and suppose
further that $({\bf C})$ holds and that $I(\cdot)$ and $K(\cdot)$
defined in \eqref{0z3} and \eqref{0z2}, respectively, are
continuous. Then,
\begin{equation*}
\vv^{-1}(\hat\theta_{n,\vv}-\theta_0)\rightarrow
I^{-1}(\theta_0)\int_0^T\Upsilon(X_t^0,\theta_0)\d B(t)~~~~\mbox{ in
probability }
\end{equation*}
as $\vv\rightarrow0$ and $n\rightarrow\8$, where  $\Upsilon(\cdot)$
is given in \eqref{0s0}.
\end{thm}

With contrast to the existing literature, the innovations of this
paper lie  in:
\begin{itemize}
\item[(i)] The classical contrast function for LSE is based on EM
algorithm. Whereas, under the monotone condition, the EM scheme   no
longer  works. Hence in the present work we adopt a tamed EM method
to establish the corresponding contrast function. The above is our
first innovation.

\item[(ii)]For the classical setup,   the discrete-time
observations at the gridpoints are sufficient to construct the
contrast function. Nevertheless, for our present model, the
discrete-time observations are insufficient to establish the
contrast function since the SDEs involved are path-dependent. In
this paper, we overcome the difficulty mentioned by linear
interpolation w.r.t. the discrete-time observations. The above is
our second innovation.

\item[(iii)] Our model is much more applicable, which allow the
coefficients to be distribution-dependent and weakly monotone. In
particular, the drift terms are allowed to be singular (e.g.,
H\"older continuous). The above is our third innovation.

\end{itemize}

Now, we provide a concrete  example to demonstrate Theorems
\ref{th1} and \ref{th2}.
\begin{exa}\label{exa}
For any $\vv\in(0,1)$, consider the following scalar
path-distribution dependent SDE
\begin{equation}
\begin{split}
\d X^\vv(t)&=\theta^{(1)}+\theta^{(2)}\int_\C\Big(-
(X^\vv(t))^3+X^\vv(t)+\int_{-r_0}^0X^\vv(t+\theta)\d\theta+\int_\C
 \zeta(\theta) \d\theta\Big)\mathscr{L}_{X^\vv_t}(\d\zeta)\\
&\quad+\vv\,\Big(1+ \int_{-r_0}^0X^\vv(t+\theta)|\d\theta\Big)\,
%\int_\C\si_0(X^\vv_t,\zeta)\mathscr{L}_{X^\vv_t}(\d\zeta)
\d B(t),~~~t\ge0
\end{split}
\end{equation}
with the initial value $X_0^\vv=\xi\in\C$  which is Lipschitz,
where, for some $c_1<c_2$ and $c_3<c_4 $,
$\theta=(\theta^{(1)},\theta^{(2)})^*\in\Theta_0:=(c_1,c_2)\times(c_3,c_4)\subset\R^2_+$
 is an unknown parameter with the
true value $\theta_0=(\theta^{(1)}_0,\theta^{(2)}_0)^*\in\Theta_0$.
Let $\hat\theta_{n,\vv}$ be the LSE for the unknown parameter
$\theta=(\theta^{(1)},\theta^{(2)})^*\in\Theta_0$. Then,
\begin{equation*}
\hat\theta_{n,\vv}\rightarrow\theta_0~~~~\mbox{ in probability as }
 \vv\rightarrow0 ~~\mbox{ and }~~n\to\8,
\end{equation*}
and
\begin{equation*}
\vv^{-1}(\hat\theta_{n,\vv}-\theta_0)\rightarrow
I^{-1}(\theta_0)\int_0^T\Upsilon(X_t^0,\theta_0)\d B(t)~~~~\mbox{ in
probability }
\end{equation*}
as $\vv\rightarrow0$ and $n\rightarrow\8$, where
\begin{equation*}
I(\theta_0)=\left(\begin{array}{ccc}
 \int_0^T\ff{1}{(1+|X_s^0|)^2}\d s  &  \int_0^T\ff{b_0(X_s^0,X_s^0)}{(1+|X_s^0|)^2}\d s \\
  \int_0^T\ff{b_0(X_s^0,X_s^0)}{(1+|X_s^0|)^2} \d s &  \int_0^T\ff{b_0^2(X_s^0,X_s^0)}{(1+|X_s^0|)^2} \d s\\
  \end{array}
  \right),
\end{equation*}
and, for $\zeta\in\C,$
\begin{equation*}
 \Upsilon(\zeta,\theta_0) = \ff{1}{1+|\zeta|}\left(\begin{array}{c}
  1 \\
  b_0(\zeta,\zeta)\\
  \end{array}
  \right).
\end{equation*}
\end{exa}

The remaining contents are organized as below: In Section
\ref{sec2}, we intend to complete the proof of Theorem \ref{th1} on
the basis of numerous auxiliary lemmas; In Section \ref{sec3}, we
aim to implement the proof  of Theorem \ref{th2}; In the final
section, we shall finish the proof of
 Example
\ref{exa}. Throughout this paper, we emphasize that  $c>0$ is a
generic constant whose value may change from line to line.

\section{Proof of Theorem \ref{th1}}\label{sec2}
To complete the proof of Theorem \ref{th1}, we provide some
technical lemmas. The lemma below expounds that the path associated
with \eqref{q3} is uniformly bounded in the $p$-th moment sense.
\begin{lem}
  Let $({\bf A1})$ and $({\bf A2})$ hold. Then,  for any $p>0$  there
is a constant $C_{p,T}>0$ such that
\begin{equation}\label{r0}
\sup_{0\le t\le T}\|X_t^0\|_\8^p\le C_{p,T}(1+\|\xi\|_\8^p),
\end{equation}
and
\begin{equation}\label{r6}
\sup_{0\le t\le T}\E\Big(\sup_{-r_0\le s\le t}|Y^\vv(s)|^p\Big)\le
C_{p,T}(1+\|\xi\|_\8^p).
\end{equation}
\end{lem}

\begin{proof}
With the assumption ({\bf A1}) in hand, the proof of \eqref{r0} can
be achieved by the chain rule and the Gronwall inequality.  We
herein omit the details since it is
 standard. Now we turn to show the argument of
\eqref{r6}. By H\"older's inequality, it suffices to verify that
\eqref{r6} holds for any $p>4.$ By It\^o's formula,   we deduce that
\begin{equation*}
\begin{split}
|Y^\vv(t)|^p&=|Y^\vv(0)|^p+\int_0^t\Big\{p|Y^\vv(s)|^{p-2}\<Y^\vv(s),b^{(\dd)}(\bar
Y^\vv_{s_\dd},\mathscr{L}_{\bar Y^\vv_{s_\dd}},\theta)\>+
\ff{p}{2}|Y^\vv(s)|^{p-2}\|\si^*(\bar Y^\vv_{s_\dd},\mathscr{L}_{\bar Y^\vv_{s_\dd}})\|^2\\
&\quad +\ff{p(p-2)}{2}|Y^\vv(s)|^{p-4}|\si^*(\bar Y^\vv_{s_\dd},\mathscr{L}_{\bar Y^\vv_{s_\dd}})Y^\vv(s)|^2\Big\}\d s\\
&\quad+p\int_0^t|Y^\vv(s)|^{p-2}\<Y^\vv(s),\si(\bar
Y^\vv_{s_\dd},\mathscr{L}_{\bar Y^\vv_{s_\dd}})\d
B(s)\>\\
%&\le\int_0^t\Big\{p|Y^\vv(s)|^{p-2}\<Y^\vv(s),b^{(\dd)}(Y^\vv_{s_\dd},\mathscr{L}_{Y^\vv_{s_\dd}},\theta)\>\\
%&\quad+\ff{p(p-1)}{2}|Y^\vv(s)|^{p-2}\|\si^*(Y^\vv_{s_\dd},\mathscr{L}_{Y^\vv_{s_\dd}})\|^2
%\Big\}\d s\\
%&\quad+p\int_0^t|Y^\vv(s)|^{p-2}\<Y^\vv(s),\si(Y^\vv_{s_\dd},\mathscr{L}_{Y^\vv_{s_\dd}})\d
%B(s)\>\\
&\le p\int_0^t|Y^\vv(s)|^{p-2}\<Y^\vv(s_\dd),b^{(\dd)}(\bar Y^\vv_{s_\dd},\mathscr{L}_{\bar Y^\vv_{s_\dd}},\theta)\>\d s\\
&\quad+p\int_0^t|Y^\vv(s)|^{p-2}\<Y^\vv(s)-Y^\vv(s_\dd),b^{(\dd)}(\bar Y^\vv_{s_\dd},\mathscr{L}_{\bar Y^\vv_{s_\dd}},\theta)\>\d s\\
&\quad+\ff{p(p-1)}{2}\int_0^t|Y^\vv(s)|^{p-2}\|\si(\bar
Y^\vv_{s_\dd},\mathscr{L}_{\bar Y^\vv_{s_\dd}})\|^2
\d s\\
&\quad+p\int_0^t|Y^\vv(s)|^{p-2}\<Y^\vv(s),\si(\bar
Y^\vv_{s_\dd},\mathscr{L}_{\bar Y^\vv_{s_\dd}})\d
B(s)\>\\
&=:\sum_{i=1}^4\Pi_i(t),~~~~~~t\in[0,T].
\end{split}
\end{equation*}
Whence, for any $t\ge0$ one has
\begin{equation}\label{w6}
\begin{split}
\Upsilon(t):=\E\Big(\sup_{-r_0\le s\le
t}|Y^\vv(s)|^p\Big)\le\|\xi\|^p_\8+\sum_{i=1}^4\E\Big(\sup_{0\le
s\le t}\Pi_i(s)\Big).
\end{split}
\end{equation}
In the sequel, we are going to claim that
\begin{equation}\label{w7}
\Upsilon(t) \le 2\|\xi\|^p_\8+ c\,t+c \int_0^t \Upsilon(s) \d s.
\end{equation}
If \eqref{w7} was true, thus \eqref{r6} follows directly from
Gronwall's inequality. So, it remains to verify that \eqref{w7}
holds true.

Let $\zeta_0(s)\equiv  {\bf 0}\in\R^d$
for any $s\in[-r_0,0].$ For  $\zeta\in\C$ and $\mu\in\mathcal
{P}_2(\C)$,   we deduce from ({\bf A1}) that
\begin{equation}\label{r1}
\begin{split}
\<\zeta(0),b(\zeta,\mu,\theta)\>&=\<\zeta(0)-\zeta_0,b(\zeta,\mu,\theta)-b(\zeta_0,\dd_{\zeta_0},\theta)\>
+\<\zeta(0),b(\zeta_0,\dd_{\zeta_0},\theta)\>\\
&\le\aa_1\|\zeta\|_\8^2+\aa_2\mathbb{W}_2(\mu,\dd_{\zeta_0})^2+|\zeta(0)|^2+|b(\zeta_0,\dd_{\zeta_0},\theta)|^2\\
&\le c\,(1+\|\zeta\|_\8^2+\mathbb{W}_2(\mu,\dd_{\zeta_0})^2),
\end{split}
\end{equation}
and from ({\bf A2}) that
\begin{equation}\label{r2}
\begin{split}
\|\si(\zeta,\mu)\|^2%&=\|\si(\zeta,\mu)-\si(\zeta_0,\dd_{\zeta_0})+\si(\zeta_0,\dd_{\zeta_0})\|^2\\
%&\le2\|\si(\zeta,\mu)-\si(\zeta_0,\dd_{\zeta_0})\|^2+2\|\si(\zeta_0,\dd_{\zeta_0})\|^2\\
&\le
2\bb_1\|\zeta\|_\8^2+2\bb_2\mathbb{W}_2(\mu,\dd_{\zeta_0})^2+2\|\si(\zeta_0,\dd_{\zeta_0})\|^2\\
&\le c\,(1+\|\zeta\|_\8^2+\mathbb{W}_2(\mu,\dd_{\zeta_0})^2).
\end{split}
\end{equation}
According to \eqref{w2}, we obtain that
\begin{equation}\label{r5}
\begin{split}
&\|\bar Y_{t_\dd}^\vv\|_\8\\
&=\max_{k=0,\cdots,M-1}\sup_{-(k+1)\dd\le s\le-k\dd}|\bar Y_{t_\dd}^\vv(s)|\\
&=\max_{k=0,\cdots,M-1}\sup_{-(k+1)\dd\le
s\le-k\dd}\Big|\ff{s+(k+1)\dd}{\dd}Y^\vv(t_\dd-k\dd)-\ff{s+k\dd}{\dd}Y^\vv(t_\dd-(k+1)\dd)\Big|\\
%&\le\max_{k=0,\cdots,M-1}\sup_{-(k+1)\dd\le
%s\le-k\dd}|Y^\vv((\lfloor
%t/\dd\rfloor-k)\dd)|\\
%&\quad+\max_{k=0,\cdots,M-1}\sup_{-(k+1)\dd\le
%s\le-k\dd}|Y^\vv((\lfloor t/\dd\rfloor-k-1)\dd)|\\
&\le 2\sup_{-r_0\le s\le t}|Y^\vv(s)|.
\end{split}
\end{equation}
Furthermore, recall the Young inequality:
\begin{equation}\label{r3}
a^\aa b^{1-\aa}\le \aa a+(1-\aa)b,~~~~~a,b\ge0,~~\aa\in[0,1],
\end{equation}
and the fundamental fact that: for any $q>0$,
\begin{equation}\label{w8}
\E|B(t)|^q\le c\, t^{q/2}.
\end{equation}
By virtue of \eqref{w2}, we notice that
\begin{equation}\label{r8}
\bar Y_{t_\dd}^\vv(0)=Y^\vv(t_\dd).
\end{equation}
Then, by exploiting \eqref{r1}, \eqref{r5} as well as \eqref{r8}, it
follows from \eqref{r3} and H\"older's inequality that
\begin{equation}\label{w3}
\begin{split}
&\E\Big(\sup_{0\le s\le t}\Pi_1(s)\Big) \\&= p\,\E\Big(\sup_{0\le
s\le t}\int_0^s\ff{|Y^\vv(u)|^{p-2}}{1+\dd^\aa |b(\bar
Y^\vv_{u_\dd},\mathscr{L}_{\bar
Y^\vv_{u_\dd}},\theta)|}\<Y^\vv(u_\dd),b(\bar
Y^\vv_{u_\dd},\mathscr{L}_{\bar Y^\vv_{u_\dd}},\theta)\>\d
u\Big)\\
&= p\,\E\Big(\sup_{0\le s\le
t}\int_0^s\ff{|Y^\vv(u)|^{p-2}}{1+\dd^\aa |b(\bar
Y^\vv_{u_\dd},\mathscr{L}_{\bar Y^\vv_{u_\dd}},\theta)|}\<\bar
Y^\vv_{u_\dd}(0),b(\bar Y^\vv_{u_\dd},\mathscr{L}_{\bar
Y^\vv_{u_\dd}},\theta)\>\d
u\Big)\\
&\le c\,\E\Big(\sup_{0\le s\le
t}\int_0^s\ff{|Y^\vv(u)|^{p-2}}{1+\dd^\aa| b(\bar
Y^\vv_{u_\dd},\mathscr{L}_{\bar
Y^\vv_{u_\dd}},\theta)|}\Big\{1+\|\bar
Y^\vv_{s_\dd}\|_\8^2+\mathbb{W}_2(\mathscr{L}_{\bar
Y^\vv_{s_\dd}},\dd_{\zeta_0})^2\Big\}\d
s\\
%&\le c\,\E\int_0^t|Y^\vv(s)|^{p-2}\Big\{1+\|\bar
%Y^\vv_{s_\dd}\|_\8^2+\mathbb{W}_2(\mathscr{L}_{\bar
%Y^\vv_{s_\dd}},\dd_{\zeta_0})^2\Big\}\d
%s\\
&\le
c\int_0^t\Big\{1+\E|Y^\vv(s)|^p+\E\|\bar Y^\vv_{s_\dd}\|_\8^p\Big\}\d s\\
&\le c\int_0^t\{1+\Upsilon(s)\}\d s.
\end{split}
\end{equation}
It is straightforward to see that, for any $\zeta\in\C$,
$\mu\in\mathcal {P}_2(\C)$, and $\theta\in\Theta$,
\begin{equation}\label{r4}
|b^{(\dd)}(\zeta,\mu,\theta)|=\ff{|b(\zeta,\mu,\theta)|}{1+\dd^\aa|b(\zeta,\mu,\theta)|}\le\dd^{-\aa}.
\end{equation}
Taking \eqref{r2}  and \eqref{r4} into consideration and making use
of \eqref{w8} and $\aa\in(0,1/2]$, for any $q\ge2$, we derive that
\begin{equation}\label{r7}
\begin{split}
\E|Y^\vv(t)-Y^\vv(t_\dd)|^q&\le
c\,\Big\{\dd^{q(1-\aa)}+\E\|\si(\bar Y^\vv_{t_\dd},\mathscr{L}_{\bar Y^\vv_{t_\dd}})\|^q\E|B(t)-B(t_\dd)|^q\Big\}\\
&\le
c\,\Big\{\dd^{q(1-\aa)}+\dd^{q/2}\E\|\si(\bar Y^\vv_{t_\dd},\mathscr{L}_{\bar Y^\vv_{t_\dd}})\|^q\Big\}\\
&\le c\,\dd^{q/2}\Big\{1+
\E\|\bar Y^\vv_{t_\dd}\|^q+\mathbb{W}_2(\mathscr{L}_{\bar Y^\vv_{t_\dd}},\dd_{\zeta_0})^q\Big\}\\
&\le c\,\dd^{q/2}\Big\{1+ \E\Big(\sup_{-r_0\le s\le
t}|Y^\vv(s)|^q\Big)\Big\},
\end{split}
\end{equation}
where in the last procedure we have used H\"older's inequality and
\eqref{r5}. Thus, taking advantage of \eqref{r4} and \eqref{r7} and
employing H\"older's inequality yields that
\begin{equation}\label{w4}
\begin{split}
\E\Big(\sup_{0\le s\le t}|\Pi_2(s)|\Big)&\le
p\E\int_0^t|Y^\vv(s)|^{p-2}|Y^\vv(s)-Y^\vv(s_\dd)|\cdot|b^{(\dd)}(\bar
Y^\vv_{s_\dd},\mathscr{L}_{\bar Y^\vv_{s_\dd}},\theta)|\d
s\\
&\le
p\dd^{-\aa}\int_0^t\E(|Y^\vv(s)|^{p-2}|Y^\vv(s)-Y^\vv(s_\dd)|)\d s\\
&\le
p\dd^{-\aa}\int_0^t\Big(\E(|Y^\vv(s)|^p)\Big)^{\ff{p-2}{p}}\Big(\E|Y^\vv(s)-Y^\vv(s_\dd)|^{\ff{p}{2}}\Big)^{\ff{2}{p}}\d
s\\
&\le
p\dd^{\ff{1}{2}-\aa}\int_0^t\Big(\E(|Y^\vv(s)|^p)\Big)^{\ff{p-2}{p}}\Big\{1+
\E\Big(\sup_{-r_0\le s\le
t}|Y^\vv(s)|^{\ff{p}{2}}\Big)\Big\}^{\ff{2}{p}}\d s\\
&\le c \int_0^t\{1+\Upsilon(s)\}\d s,
\end{split}
\end{equation}
where  in the last display we  used
 $\aa\in(0,1/2]$ and  \eqref{r3}. Next, we observe
that
\begin{equation}\label{q1}
\begin{split}
\E\Big(\sup_{0\le s\le t}\Pi_3(s)\Big)&\le
\ff{p(p-1)}{2}\int_0^t\E(|Y^\vv(s)|^{p-2}\|\si(\bar
Y^\vv_{s_\dd},\mathscr{L}_{\bar Y^\vv_{s_\dd}})\|^2) \d s.
\end{split}
\end{equation}
Using Burkhold-Davis-Gundy's (BDG's for short) inequality and
\eqref{r3}, we infer that
\begin{equation}\label{q2}
\begin{split}
\E\Big(\sup_{0\le s\le t}\Pi_4(s)\Big)&\le p\,\E\Big(\sup_{0\le s\le
t}\Big|\int_0^s|Y^\vv(u)|^{p-2}\<Y^\vv(u),\si(\bar
Y^\vv_{u_\dd},\mathscr{L}_{\bar Y^\vv_{u_\dd}})\d
B(u)\>\Big|\Big)\\
&\le 4\ss2\,p\,\E\Big(\int_0^t|Y^\vv(s)|^{2(p-2)}|\si^*(\bar
Y^\vv_{s_\dd},\mathscr{L}_{\bar Y^\vv_{s_\dd}})Y^\vv(s)|^2\d
s\Big)^{1/2}\\
&\le 4\ss2\,p\,\E\Big(\sup_{0\le s\le
t}|Y^\vv(s)|^p\int_0^t|Y^\vv(s)|^{p-2}\|\si(\bar
Y^\vv_{s_\dd},\mathscr{L}_{\bar Y^\vv_{s_\dd}})\|^2\d
s\Big)^{1/2}\\
&\le\ff{1}{2}\Upsilon(t)+16p^2\int_0^t\E(|Y^\vv(s)|^{p-2}\|\si(\bar
Y^\vv_{s_\dd},\mathscr{L}_{\bar Y^\vv_{s_\dd}})\|^2)\d s.
\end{split}
\end{equation}
Subsequently, one gets from \eqref{q1} and \eqref{q2} that
\begin{equation}\label{w5}
\begin{split}
&\E\Big(\sup_{0\le s\le t}\Pi_3(s)\Big)+\E\Big(\sup_{0\le s\le
t}\Pi_4(s)\Big)\\&\le
\ff{1}{2}\Upsilon(t)+c\int_0^t\E(|Y^\vv(s)|^{p-2}\|\si(\bar
Y^\vv_{s_\dd},\mathscr{L}_{\bar Y^\vv_{s_\dd}})\|^2)\d
s\\
&\le \ff{1}{2}\Upsilon(t)+ c\int_0^t\{\E|Y^\vv(s)|^p+\E\|\si(\bar
Y^\vv_{s_\dd},\mathscr{L}_{\bar Y^\vv_{s_\dd}})\|^p\}
\d s\\
&\le \ff{1}{2}\Upsilon(t)+ c\int_0^t\Big\{1+\E|Y^\vv(s)|^p+\E\|\bar
Y^\vv_{s_\dd}\|_\8^p+\mathbb{W}_2(\mathscr{L}_{\bar
Y^\vv_{s_\dd}},\dd_{\zeta_0})^p\Big\}
\d s\\
&\le \ff{1}{2}\Upsilon(t)+ c \int_0^t\{1+\Upsilon(s)\}\d s,
\end{split}
\end{equation}
where  we have adopted \eqref{r3} in the second inequality, used
\eqref{r2} in the last two step, and utilized H\"older's inequality,
in addition to \eqref{r5}, in the last procedure. Substituting
\eqref{w3}, \eqref{w4}, and \eqref{w5} into \eqref{w6} gives that
\begin{equation*}
\Upsilon(t) \le \|\xi\|^p_\8+ \ff{1}{2}\Upsilon(t)+ c
\int_0^t\{1+\Upsilon(s)\}\d s.
\end{equation*}
As a consequence, \eqref{w7} is now available.
\end{proof}

The following lemma shows that the linear interpolation $\bar
Y^\vv_{t_\dd}$ approaches $X_t^0$ in the mean square sense as %the
%quantities
 $\vv$ and $\dd$ go to zero.
\begin{lem}\label{le1}
  Assume $({\bf A1}), ({\bf A2}), ({\bf B1})$ and $({\bf B4})$. Then,
for any $\bb\in(0,1)$, there exists $c_\bb>0$

\begin{equation}\label{a9}
 \sup_{0\le t\le T}\E\|\bar Y^\vv_{t_\dd}-X_t^0\|_\8^2\le
 c_\bb(\dd^\bb+\vv^2+\dd^{2\aa}),
\end{equation}
where $\aa\in(0,1/2]$ is introduced in \eqref{e4}.
\end{lem}

\begin{proof}
For any  $\bb\in(0,1)$ and $t\in[0,T]$, by H\"older's inequality and
$Y_0^\vv=X_0^0=\xi$, we find that
\begin{equation}\label{a7}
\begin{split}
\E\|\bar Y^\vv_{t_\dd}-X_t^0\|_\8^2&\le3\,\E\|Y_t^\vv-\bar Y^\vv_{t_\dd}\|_\8^2+3\,\E\|Y^\vv_t-X_t^\vv\|_\8^2+3\,\E\|X^\vv_t-X_t^0\|_\8^2\\
&\le3\,\E\Big(\sup_{-r_0\le v\le 0}|Y^\vv(t+v)-\bar
Y^\vv_{t_\dd}(v)|^2\Big)+3\,\E\Big(\sup_{0\le s\le
t}|Y^\vv(s)-X^\vv(s)|^2\Big)\\
&\quad+3\,\E\Big(\sup_{0\le s\le t}|X^\vv(s)-X^0(s)|^2\Big)\\
&\le3\,
M^{1-\bb}\max_{k=0,\cdots,M-1}\,\Big(\E\Big(\sup_{-(k+1)\dd\le
v\le-k\dd}| Y^\vv(t+v)-\bar Y^\vv_{t_\dd}(v)|^{\ff{2}{1-\bb}}\Big)\Big)^{1-\bb}\\
&\quad+3\,\E\Big(\sup_{0\le s\le t}|Y^\vv(s)-X^\vv(s)|^2\Big)
+3\,\E\Big(\sup_{0\le s\le
t}|X^\vv(s)-X^0(s)|^2\Big)\\
&=:\LL_1(t,\vv,\dd)+\LL_2(t,\vv,\dd)+\LL_3(t,\vv,\dd),
\end{split}
\end{equation}
where $M>0$ such that $M\dd=r_0.$ Hereinafter, we intend to estimate
$\LL_i(t,\vv,\dd)$, $i=1,2,3,$ respectively. In the first place, we
shall show that
\begin{equation}\label{e3}
\LL_1(t,\vv,\dd)\le c\,\dd^\bb,~~~~t\in[0,T].
\end{equation}
For   $t\in[0,T)$, there is an integer $k_0\ge0$ such that
$t\in[k_0\dd,(k_0+1)\dd).$ From \eqref{w2}, it follows that
\begin{equation}\label{a00}
\begin{split}
&\LL_1(t,\vv,\dd)\\&\le c\,
M^{1-\bb}\max_{k=0,\cdots,M-1}\,\Big(\E\Big(\sup_{(k_0-k-1)\dd\le
s\le(k_0+1-k)\dd}|
Y^\vv(s)-Y^\vv((k_0-k)\dd)|^{\ff{2}{1-\bb}}\Big)\Big)^{1-\bb}\\
&\quad+
 c\,
M^{1-\bb}\max_{k=0,\cdots,M-1}\,\Big(\E\Big(\sup_{(k_0-k-1)\dd\le
s\le(k_0+1-k)\dd}|
Y^\vv(s)-Y^\vv((k_0-k-1)\dd)|^{\ff{2}{1-\bb}}\Big)\Big)^{1-\bb}\\
&\le
 c\,
M^{1-\bb}\max_{k=0,\cdots,M-1}\,\Big(\E\Big(\sup_{(k_0-k-1)\dd\le
s\le(k_0+1-k)\dd}|
Y^\vv(s)-Y^\vv((k_0-k-1)\dd)|^{\ff{2}{1-\bb}}\Big)\Big)^{1-\bb}\\
&\quad+c\, M^{1-\bb}\max_{k=0,\cdots,M-1}\,\Big(\E|
Y^\vv((k_0-k)\dd)-Y^\vv((k_0-k-1)\dd)|^{\ff{2}{1-\bb}}\Big)^{1-\bb}.
% &=:\LL_{11}(t,\vv,\dd)+\LL_{12}(t,\vv,\dd).
\end{split}
\end{equation}
In case of  $k\ge k_0+1$, by virtue of ({\bf B4}),  one has
\begin{equation*}
\LL_1(t,\vv,\dd)\le c\,M^{1-\bb}\dd^2\le c\,r_0^{1-\bb}\dd^\bb.
\end{equation*}
In terms of ({\bf B1}), for any $\zeta\in\C$ and $\mu\in\mathcal
{P}_2(\C)$,
\begin{equation}\label{q4}
\begin{split}
|b(\zeta,\mu,\theta_0)|&\le|b(\zeta,\mu,\theta_0)-b(\zeta_0,\dd_{\zeta_0},\theta_0)|+|b(\zeta_0,\dd_{\zeta_0},\theta_0)|\\
&\le
L_1\Big\{(1+\|\zeta\|_\8^{q_1})\|\zeta\|_\8+\mathbb{W}_2(\mu,\dd_{\zeta_0})\Big\}+|b(\zeta_0,\dd_{\zeta_0},\theta_0)|.
\end{split}
\end{equation}
Let  $k'\ge0$ be an arbitrary integer. For any
$t\in[k'\dd,(k'+2)\dd]$, note from BDG's inequality followed by
H\"older's inequality that
\begin{equation*}
\begin{split}
&\E\Big(\sup_{k'\dd\le t\le
(k'+2)\dd}|Y^\vv(t)-Y^\vv(k'\dd)|^{\ff{2}{1-\bb}}\Big)\\&\le
c\,\E\Big(\int_{k'\dd}^{(k'+2)\dd}|b^{(\dd)}(\bar
Y^\vv_{s_\dd},\mathscr{L}_{\bar Y^\vv_{s_\dd}},\theta_0)|\d
s\Big)^{\ff{2}{1-\bb}}+c\,\E\Big(\sup_{k'\dd\le t\le
(k'+2)\dd}\Big|\int_{k'\dd}^t\si(\bar
Y^\vv_{s_\dd},\mathscr{L}_{\bar Y^\vv_{s_\dd}})\d
B(s)\Big|^{\ff{2}{1-\bb}}\Big)\\
&\le c\,\E\Big(\int_{k'\dd}^{(k'+2)\dd}|b^{(\dd)}(\bar
Y^\vv_{s_\dd},\mathscr{L}_{\bar Y^\vv_{s_\dd}},\theta_0)|\d
s\Big)^{\ff{2}{1-\bb}}+c\,\E\Big(\int_{k'\dd}^{(k'+2)\dd}\|\si(\bar
Y^\vv_{s_\dd},\mathscr{L}_{\bar Y^\vv_{s_\dd}})\|^2\d
s\Big)^{\ff{1}{1-\bb}}\\
&\le c\,\dd^{\ff{\bb}{1-\bb}}\int_{k'\dd}^{(k'+2)\dd}\Big\{\E|b(\bar
Y^\vv_{s_\dd},\mathscr{L}_{\bar
Y^\vv_{s_\dd}},\theta_0)|^{\ff{2}{1-\bb}}+ \E\|\si(\bar
Y^\vv_{s_\dd},\mathscr{L}_{\bar
Y^\vv_{s_\dd}})\|^{\ff{2}{1-\bb}}\Big\}\d s,
\end{split}
\end{equation*}
where in the last display we have used the fact that
\begin{equation}\label{t7}
|b^{(\dd)}(\zeta,\mu,\theta_0)|\le|b(\zeta,\mu,\theta_0)|,~~~~\zeta\in\C,~~~\mu\in\mathcal
{P}_2(\C).
\end{equation}
Subsequently, taking  \eqref{r6}, \eqref{r2} and \eqref{q4} into
account and making use of H\"older's inequality  yields that
\begin{equation}\label{d1}
\begin{split}
&\E\Big(\sup_{k'\dd\le t\le
(k'+2)\dd}|Y^\vv(t)-Y^\vv(k'\dd)|^{\ff{2}{1-\bb}}\Big)\\
&\le
c\,\dd^{\ff{\bb}{1-\bb}}\int_{k'\dd}^{(k'+2)\dd}\Big\{1+\E\|\bar
Y^\vv_{s_\dd}\|_\8^{\ff{2(1+q_1)}{1-\bb}}+
\mathbb{W}_2(\mathscr{L}_{\bar
Y^\vv_{s_\dd}},\dd_{\zeta_0})^{\ff{2}{1-\bb}}\Big\}\d
s\\
&\le
c\,\dd^{\ff{\bb}{1-\bb}}\int_{k'\dd}^{(k'+2)\dd}\Big\{1+\E\|\bar
Y^\vv_{s_\dd}\|_\8^{\ff{2(1+q_1)}{1-\bb}}\Big\}\d
s\\
&\le c\,\dd^{\ff{1}{1-\bb}}.
\end{split}
\end{equation}
Hence,    it follows from  \eqref{a00} and \eqref{d1} with
$k^\prime=k_0-k-1$ that
\begin{equation*}
\begin{split}
\LL_1(t,\vv,\dd)\le c\, M^{1-\bb}\dd\le c\,\dd^\bb
\end{split}
\end{equation*}
provided that $k\le k_0-1$. Whenever $k=k_0$, we deduce from
\eqref{a00}, \eqref{d1} with $k'=0$ as well as ({\bf B4}) that
\begin{equation*}
\begin{split}
\LL_1(t,\vv,\dd)&\le c\, M^{1-\bb}\Big(\E\Big(\sup_{0\le s\le\dd}|
Y^\vv(s)-Y^\vv(0)|^{\ff{2}{1-\bb}}\Big)\Big)^{1-\bb}\\
&\quad+ c\, M^{1-\bb}\Big(\E\Big(\sup_{-\dd\le s\le0}|
Y^\vv(s)-Y^\vv(-\dd)|^{\ff{2}{1-\bb}}\Big)\Big)^{1-\bb}\\
&\quad+c\, M^{1-\bb}|
Y^\vv(0)-Y^\vv(-\dd)|^2\\
&\le c\, M^{1-\bb}\dd\\
&\le c\,\dd^\bb.
\end{split}
\end{equation*}
Next, we are going to claim that
\begin{equation}\label{a8}
\LL_3(t,\vv,\dd) \le c\,\vv^2,~~~~t\in[0,T].
\end{equation}
Following the argument to derive \eqref{r6}, we deduce that, for
some constant $C_{p,T}>0,$
\begin{equation}\label{0r6}
\sup_{0\le t\le T}\E\|X^\vv_t\|_\8^p\le C_{p,T}(1+\|\xi\|_\8^p).
\end{equation}
By the It\^o formula  and $X^\vv_0=X_0^0=\xi$, we observe that
\begin{equation*}
\begin{split}
&|X^\vv(t)-X^0(t)|^2\\
&=\int_0^t\{2\<X^\vv(s)-X^0(s),b(X_s^\vv,\mathscr{L}_{X_s^\vv},\theta_0)-b(X_s^0,\mathscr{L}_{X_s^0},\theta_0)\>+\vv^2\|
\si(X_s^\vv,\mathscr{L}_{X_s^\vv})\|^2\}\d
s\\
&\quad+2\,\vv\int_0^t\<X^\vv(s)-X^0(s),\si(X_s^\vv,\mathscr{L}_{X_s^\vv})\d
B(s)\>.
\end{split}
\end{equation*}
Thus, by using BDG's inequality and  \eqref{r3} and noting that
$X^\vv_0=X_0^0=\xi$, we infer from ({\bf A1}) and \eqref{r2} that
\begin{equation*}
\begin{split}
\LL_3(t,\vv,\dd)&\le2\int_0^t\{\aa_1\E\|X_s^\vv-X_s^0\|_\8^2+\aa_2\mathbb{W}_2(\mathscr{L}_{X_s^\vv},\mathscr{L}_{X_s^0})^2\}\d
s\\
&\quad+c\,\vv^2\int_0^t\{1+\E\|X_s^\vv\|_\8^2+\mathbb{W}_2(\mathscr{L}_{X_s^\vv},\dd_{\zeta_0})^2\}\d
s\\
&\quad+8\ss2\,\vv\E\Big(\int_0^t|\si^*(X_s^\vv,\mathscr{L}_{X_s^\vv})(X^\vv(s)-X^0(s))|^2\d
s\Big)^{1/2}\\
&\le c\int_0^t\LL_3(s,\vv,\dd)\d s+c\,\vv^2\int_0^t\{1+\E\|X_s^\vv\|_\8^2\}\d s\\
&\quad+8\ss2\,\vv\E\Big(\sup_{0\le s\le
t}|X^\vv(s)-X^0(s)|^2\int_0^t\|\si(X_s^\vv,\mathscr{L}_{X_s^\vv})\|^2\d
s\Big)^{1/2}\\
&\le\ff{1}{2}\LL_3(t,\vv,\dd)+c\int_0^t\LL_3(s,\vv,\dd)\d
s+c\,\vv^2\int_0^t\{1+\E\|X_s^\vv\|_\8^2\}\d s.
\end{split}
\end{equation*}
So, one has
\begin{equation*}
\LL_3(t,\vv,\dd) \le c\int_0^t\LL_3(s,\vv,\dd)\d
s+c\,\vv^2\int_0^t\{1+\E\|X_s^\vv\|_\8^2\}\d s.
\end{equation*}
Thus, \eqref{a8} follows from \eqref{0r6} and Gronwall's inequality.
Finally, we intend to verify that
\begin{equation}\label{s1}
\LL_2(t,\vv,\dd)\le c\,(\dd^\bb+\dd^{2\aa}),~~~~t\in[0,T].
\end{equation}
Also, by It\^o's formula, we derive from $X_0^\vv=Y_0^\vv=\xi$ that
\begin{align*}
|X^\vv(t)-Y^\vv(t)|^2&=2\int_0^t\<X^\vv(s)-Y^\vv(s),
b(X_s^\vv,\mathscr{L}_{X_s^\vv},\theta_0)-b(Y_s^\vv,\mathscr{L}_{Y_s^\vv},\theta_0)\>\d s\\
&\quad+2\int_0^t\<X^\vv(s)-Y^\vv(s),
b(Y_s^\vv,\mathscr{L}_{Y_s^\vv},\theta_0)-b(\bar Y_{s_\dd}^\vv,\mathscr{L}_{\bar Y_{s_\dd}^\vv},\theta_0)\>\d s\\
&\quad+2\int_0^t\<X^\vv(s)-Y^\vv(s),
b(\bar Y_{s_\dd}^\vv,\mathscr{L}_{\bar Y_{s_\dd}^\vv},\theta_0)-b^{(\dd)}(\bar Y_{s_\dd}^\vv,\mathscr{L}_{\bar Y_{s_\dd}^\vv},\theta_0)\>\d s\\
&\quad+\vv^2\int_0^t\|\si(X_s^\vv,\mathscr{L}_{X_s^\vv})-\si(\bar
Y_{s_\dd}^\vv,\mathscr{L}_{\bar Y_{s_\dd}^\vv})\|^2\d
s\\
&\quad+2\,\vv\int_0^t\<X^\vv(s)-Y^\vv(s),(\si(X_s^\vv,\mathscr{L}_{X_s^\vv})-\si(\bar
Y_{s_\dd}^\vv,\mathscr{L}_{\bar Y_{s_\dd}^\vv}))\d
B(s)\>\\
&=:\Xi_1(t)+\Xi_2(t)+\Xi_3(t)+\Xi_4(t)+\Xi_5(t).
\end{align*}
In view of ({\bf A1}), we deduce   that
\begin{equation}\label{a22}
\begin{split}
\E\Big(\sup_{0\le s\le
t}\Xi_1(s)\Big)&\le2\int_0^t\{\aa_1\E\|X_s^\vv-Y_s^\vv\|_\8^2+\aa_2\mathbb{W}_2(\mathscr{L}_{X_s^\vv},\mathscr{L}_{Y_s^\vv})^2\}\d
s\\
&\le c\,\int_0^t\E\|X_s^\vv-Y_s^\vv\|_\8^2\d s\\
&\le c\int_0^t\LL_2(s,\vv,\dd)\d s.
\end{split}
\end{equation}
Carrying out a similar argument to derive \eqref{e3}, for any
$\kk>2$, we have
\begin{equation}\label{w0}
\sup_{0\le t\le T}\E\|Y^\vv_t-\bar Y_{t_\dd}^\vv\|_\8^\kk\le
c\,\dd^{\ff{\kk}{2}-1}.
\end{equation}
Taking  ({\bf A1}),  \eqref{r6} and \eqref{w0} into consideration
and applying H\"older's inequality that
\begin{equation}\label{a2}
\begin{split}
&\E\Big(\sup_{0\le s\le
t}|\Xi_2(s)|\Big)\\&\le\int_0^t\{\E|X^\vv(s)-Y^\vv(s)|^2+
\E|b(Y_s^\vv,\mathscr{L}_{Y_s^\vv},\theta_0)-b(\bar
Y_{s_\dd}^\vv,\mathscr{L}_{\bar Y_{s_\dd}^\vv},\theta_0)|^2\}\d
s\\
&\le c\int_0^t\E|X^\vv(s)-Y^\vv(s)|^2\d
s\\
&\quad+c\int_0^t\E\{(1+\| Y_s^\vv\|_\8^{2q_1}+\|\bar
Y_{s_\dd}^\vv\|_\8^{2q_1})\|Y^\vv_s-\bar
Y_{s_\dd}^\vv\|_\8^2+\mathbb{W}_2(\mathscr{L}_{Y_s^\vv},\mathscr{L}_{\bar
Y_{s_\dd}^\vv})^2\} \d
s\\
&\le c\int_0^t\LL_2(s,\vv,\dd)\d s+c\int_0^t\Big(\E\|Y^\vv_s-\bar
Y_{s_\dd}^\vv\|_\8^{\ff{2}{1-\bb}}\Big)^{1-\bb}\d s\\
&\quad+c\int_0^t\Big(\E\Big(1+\| Y_s^\vv\|_\8^{2q_1}+\|\bar
Y_{s_\dd}^\vv\|_\8^{2q_1}\Big)^{\ff{1}{\bb}}\Big)^{\bb}\Big(\E\|Y^\vv_s-\bar
Y_{s_\dd}^\vv\|_\8^{\ff{2}{1-\bb}}\Big)^{1-\bb}\d
s\\
&\le c\,\dd^\bb+c\int_0^t\LL_2(s,\vv,\dd)\d s.
\end{split}
\end{equation}
According to \eqref{e4} and in view of \eqref{r6} and \eqref{q4}, it
follows from H\"older's inequality that
\begin{equation}\label{a3}
\begin{split}
&\E\Big(\sup_{0\le s\le
t}|\Xi_3(s)|\Big)\\&\le2\int_0^t\E\{|X^\vv(s)-Y^\vv(s)|\cdot
|b(\bar Y_{s_\dd}^\vv,\mathscr{L}_{\bar Y_{s_\dd}^\vv},\theta_0)-b^{(\dd)}(\bar Y_{s_\dd}^\vv,\mathscr{L}_{\bar Y_{s_\dd}^\vv},\theta_0)|\}\d s\\
&\le c\int_0^t\E\Big\{|X^\vv(s)-Y^\vv(s)|^2+
 \ff{\dd^{2\aa}|b(\bar Y_{s_\dd}^\vv,\mathscr{L}_{\bar Y_{s_\dd}^\vv},\theta_0)|^4}{(1+\dd^\aa|
 b(\bar Y_{s_\dd}^\vv,\mathscr{L}_{\bar Y_{s_\dd}^\vv},\theta_0)|)^2}
 \Big\}\d s\\
 &\le c\int_0^t\Big\{\E|X^\vv(s)-Y^\vv(s)|^2+
\dd^{2\aa}\E|b(\bar Y_{s_\dd}^\vv,\mathscr{L}_{\bar
Y_{s_\dd}^\vv},\theta_0)|^4
 \Big\}\d s\\
 &\le c\int_0^t\Big\{\E|X^\vv(s)-Y^\vv(s)|^2+
\dd^{2\aa}\{1+\E\|\bar
Y_{s_\dd}^\vv\|_\8^{4(1+q_1)}+\mathbb{W}_2(\mathscr{L}_{\bar
Y_{s_\dd}^\vv},\dd_{\zeta_0})^4\}
 \Big\}\d s\\
 &\le c\,\dd^{2\aa}+c\int_0^t\LL_2(s,\vv,\dd)\d s.
\end{split}
\end{equation}
 Next, owing to $\vv\in(0,1)$, $({\bf A2})$, and
\eqref{e3}, one gets that
\begin{equation}\label{a1}
\begin{split}
\E\Big(\sup_{0\le s\le t}\Xi_4(s)\Big)&\le
c\int_0^t\{\E\|X_s^\vv-\bar
Y_{s_\dd}^\vv\|_\8^2+\mathbb{W}_2(\mathscr{L}_{X_s^\vv},\mathscr{L}_{\bar
Y_{s_\dd}^\vv})^2\}\d
s\\
&\le c\int_0^t\{\E\|X_s^\vv-Y_s^\vv\|_\8^2+\E\|Y_s^\vv-\bar
Y_{s_\dd}^\vv\|_\8^2\}\d
s\\
&\le c\int_0^t\{\LL_1(s,\vv,\dd)+\LL_2(s,\vv,\dd)\}\d
s\\
&\le c\,\dd^\bb+c\int_0^t\LL_2(s,\vv,\dd)\d s.
\end{split}
\end{equation}
Next, for $\vv\in(0,1)$, BDG's inequality and Young's inequality
\eqref{r3}, besides \eqref{a1},  give that
\begin{equation}\label{a5}
\begin{split}
&\E\Big(\sup_{0\le s\le
t}|\Xi_5(s)|\Big)\\&\le8\ss2\,\E\Big(\int_0^t
|(\si(X_s^\vv,\mathscr{L}_{X_s^\vv})-\si(\bar
Y_{s_\dd}^\vv,\mathscr{L}_{\bar
Y_{s_\dd}^\vv}))^*(X^\vv(s)-Y^\vv(s))|^2\d
s\Big)^{1/2}\\
&\le8\ss2\,\E\Big(\sup_{0\le\le
t}|X^\vv(s)-Y^\vv(s)|^2\int_0^t\|\si(X_s^\vv,\mathscr{L}_{X_s^\vv})-\si(\bar
Y_{s_\dd}^\vv,\mathscr{L}_{\bar Y_{s_\dd}^\vv})\|^2\d
s\Big)^{1/2}\\
&\le
\ff{1}{2}\LL_2(t,\vv,\dd)+c\int_0^t\E\|\si(X_s^\vv,\mathscr{L}_{X_s^\vv})-\si(\bar
Y_{s_\dd}^\vv,\mathscr{L}_{\bar Y_{s_\dd}^\vv})\|^2\d
s\\
&\le
\ff{1}{2}\LL_2(t,\vv,\dd)+c\,\dd^\bb+c\int_0^t\LL_2(s,\vv,\dd)\d s.
\end{split}
\end{equation}
Thus, \eqref{a22}, \eqref{a2}-\eqref{a5} yield that
\begin{equation*}
\LL_2(t,\vv,\dd)\le\ff{1}{2}\LL_2(t,\vv,\dd)+
c\,(\dd^\bb+\dd^{2\aa})+c\int_0^t\LL_2(s,\vv,\dd)\d s.
\end{equation*}
Namely,
\begin{equation*}
\LL_2(t,\vv,\dd)\le
c\,(\dd^\bb+\dd^{2\aa})+c\int_0^t\LL_2(s,\vv,\dd)\d s.
\end{equation*}
As a result, we obtain from Gronwall's inequality that
\begin{equation}\label{a6}
\LL_2(t,\vv,\dd)\le c\,(\dd^\bb+\dd^{2\aa}).
\end{equation}
Inserting \eqref{e3}, \eqref{a8}, and \eqref{a6} back into
\eqref{a7} leads to the desired assertion \eqref{a9}.
\end{proof}

\begin{rem}
{\rm The convergence rate of EM scheme for path-independent SDEs
under the global Lipschitz condition is one half. Taking $\aa=1/2$
in \eqref{s1}, we conclude that the convergence rate of  the tamed
EM scheme constructed in \eqref{q3} is close sufficiently to one
half. This demonstrate the distinct features between path-dependent
SDEs and path-independent SDEs.  }
\end{rem}

The lemma below plays a crucial role in revealing the asymptotic
behavior of the LSE of the unknown parameter $\theta\in\Theta$.

\begin{lem}\label{lem}
  Let $({\bf A1})-({\bf A2})$ and $({\bf B1})-({\bf B4})$ hold. Then,
\begin{equation}\label{t4}
\begin{split}
&\dd\sum_{k=1}^n(\Gamma^{(\dd)})^*(\bar Y_{(k-1)\dd}^\vv,\theta,\theta_0)\hat\si(\bar Y_{(k-1)\dd}^\vv)\Gamma^{(\dd)}(\bar Y_{(k-1)\dd}^\vv,\theta,\theta_0)\\
&\rightarrow\Xi(\theta):=\int_0^T\Gamma(X_t^0,\theta,\theta_0)^*\hat\si(X_t^0)\Gamma(X_s^0,\theta,\theta_0)\d
t
\end{split}
\end{equation}
in $L^1$ as $\vv\rightarrow0$ and $\delta\rightarrow0 $ $($i.e.,
$n\rightarrow\8$$)$. Moreover,
\begin{equation}\label{q6}
 \sum_{k=1}^n(\Gamma^{(\dd)})^*(\bar Y_{(k-1)\dd}^\vv,\theta,\theta_0)\hat\si(\bar Y_{(k-1)\dd}^\vv)P_k(\theta_0) \rightarrow0
\end{equation}
in $L^2$ as $\vv\rightarrow0$.
\end{lem}

\begin{proof}
Observe that
\begin{equation*}
\begin{split}
&\dd\sum_{k=1}^n(\Gamma^{(\dd)})^*(\bar
Y_{(k-1)\dd}^\vv,\theta,\theta_0)\hat\si(\bar
Y_{(k-1)\dd}^\vv)\Gamma^{(\dd)}(\bar
Y_{(k-1)\dd}^\vv,\theta,\theta_0)-
\int_0^T\Gamma^*(X_t^0,\theta,\theta_0)\hat\si(X_t^0)\Gamma(X_t^0,\theta,\theta_0)\d
t\\
&=\int_0^T\Big\{(\Gamma^{(\dd)})^*(\bar
Y_{t_\dd}^\vv,\theta,\theta_0)\hat\si(\bar
Y_{t_\dd}^\vv)\Gamma^{(\dd)}(\bar Y_{t_\dd}^\vv,\theta,\theta_0)
-\Gamma^*(X_t^0,\theta,\theta_0)\hat\si(X_t^0)\Gamma(X_t^0,\theta,\theta_0)\Big\}\d
t\\
&=\int_0^T \Big(\Gamma^{(\dd)}(\bar
Y_{t_\dd}^\vv,\theta,\theta_0)-\Gamma(X_t^0,\theta,\theta_0)\Big)^*\hat\si(\bar
Y_{t_\dd}^\vv)\Gamma^{(\dd)}(\bar Y_{t_\dd}^\vv,\theta,\theta_0)\d t\\
&\quad+\int_0^T\Gamma(X_t^0,\theta,\theta_0)^*
 \Big(\hat\si(\bar Y_{t_\dd}^\vv)-\hat\si(X_t^0)\Big)\Gamma^{(\dd)}(\bar
Y_{t_\dd}^\vv,\theta,\theta_0) \d
t\\
&\quad+\int_0^T\Gamma(X_t^0,\theta,\theta_0)^*
\hat\si(X_t^0)\Big(\Gamma^{(\dd)}(\bar
Y_{t_\dd}^\vv,\theta,\theta_0)-\Gamma(X_t^0,\theta,\theta_0)\Big) \d
t\\
&=:J_1(\vv,\dd)+J_2(\vv,\dd)+J_3(\vv,\dd).
\end{split}
\end{equation*}
From ({\bf B1}) and \eqref{q4}, a direct calculation shows  that,
for any random variables $\zeta_1,\zeta\in\C$ with
$\mathscr{L}_{\zeta_1},\mathscr{L}_{\zeta_2}\in\mathcal {P}_2(\C)$,
\begin{equation}\label{t1}
\begin{split}
&|\Gamma^{(\dd)}(\zeta_1,\theta,\theta_0)-\Gamma(\zeta_2,\theta,\theta_0)|\\
&=|b^{(\dd)}(\zeta_1,\mathscr{L}_{\zeta_1},\theta_0)-b(\zeta_2,\mathscr{L}_{\zeta_2},\theta_0)+
b(\zeta_2,\mathscr{L}_{\zeta_2},\theta)-b^{(\dd)}(\zeta_1,\mathscr{L}_{\zeta_1},\theta)
|\\
&\le|b(\zeta_1,\mathscr{L}_{\zeta_1},\theta_0)-b(\zeta_2,\mathscr{L}_{\zeta_2},\theta_0)|
+|b(\zeta_2,\mathscr{L}_{\zeta_2},\theta)-b(\zeta_1,\mathscr{L}_{\zeta_1},\theta)|\\
&\quad+|b^{(\dd)}(\zeta_1,\mathscr{L}_{\zeta_1},\theta_0)-b(\zeta_1,\mathscr{L}_{\zeta_1},\theta_0)|
+|b(\zeta_1,\mathscr{L}_{\zeta_1},\theta)-b^{(\dd)}(\zeta_1,\mathscr{L}_{\zeta_1},\theta)|\\
&=|b(\zeta_1,\mathscr{L}_{\zeta_1},\theta_0)-b(\zeta_2,\mathscr{L}_{\zeta_2},\theta_0)|
+|b(\zeta_2,\mathscr{L}_{\zeta_2},\theta)-b(\zeta_1,\mathscr{L}_{\zeta_1},\theta)|\\
&\quad+\dd^\aa\Big|\ff{|b(\zeta_1,\mathscr{L}_{\zeta_1},\theta_0)|}{1+\dd^\aa|b(\zeta_1,\mathscr{L}_{\zeta_1},\theta_0)|}b(\zeta_1,\mathscr{L}_{\zeta_1},\theta_0)\Big|
+\dd^\aa\Big|\ff{|b(\zeta_1,\mathscr{L}_{\zeta_1},\theta)|}{1+\dd^\aa|b(\zeta_1,\mathscr{L}_{\zeta_1},\theta)|}
b(\zeta_1,\mathscr{L}_{\zeta_1},\theta)\Big|\\
&\le|b(\zeta_1,\mathscr{L}_{\zeta_1},\theta_0)-b(\zeta_2,\mathscr{L}_{\zeta_2},\theta_0)|
+|b(\zeta_2,\mathscr{L}_{\zeta_2},\theta)-b(\zeta_1,\mathscr{L}_{\zeta_1},\theta)|\\
&\quad+\dd^{\aa}\{|b(\zeta_1,\mathscr{L}_{\zeta_1},\theta_0)|^2+|b(\zeta_1,\mathscr{L}_{\zeta_1},\theta)|^2\}\\
&\le
c\Big\{(1+\|\zeta_1\|_\8^{q_1}+\|\zeta_2\|_\8^{q_1})\|\zeta_1-\zeta_2\|_\8+\mathbb{W}_2(\mathscr{L}_{\zeta_1},\mathscr{L}_{\zeta_2})\Big\}\\
&\quad+c\,\dd^{\aa}\Big\{1+\|\zeta_1\|_\8^{2(1+q_1)}+\mathbb{W}_2(\mathscr{L}_{\zeta_1},\dd_{\zeta_0})^2\Big\}.
\end{split}
\end{equation}
Next, for a  random variable  $\zeta\in\C$ with
$\mathscr{L}_\zeta\in\mathcal {P}_2(\C)$, by  \eqref{q4} and
\eqref{t7}, it follows that
\begin{equation}\label{t3}
\begin{split}
|\Gamma^{(\dd)}(\zeta,\theta,\theta_0)|+|\Gamma(\zeta,\theta,\theta_0)|%&\le2|b(\zeta,\mathscr{L}_\zeta,\theta_0)|
%+2|b(\zeta,\mathscr{L}_\zeta,\theta)|
\le c\,\Big\{1+
 \|\zeta\|_\8^{1+q_1} +
 \mathbb{W}_2(\mathscr{L}_\zeta,\dd_{\zeta_0})\Big\}
 \end{split}
\end{equation}
and, due to ({\bf B3}), that
\begin{equation}\label{t2}
\|\hat\si(\zeta)\|\le\|\hat\si(\zeta)-\hat\si(0)\|+\|\hat\si(0)\|
\le c\,\Big\{1+
 \|\zeta\|_\8^{1+q_3} +
 \mathbb{W}_2(\mathscr{L}_\zeta,\dd_{\zeta_0})\Big\}.
\end{equation}
Consequently, combining   \eqref{t1} with \eqref{t3} and\eqref{t2},
for $q:=q_1\vee q_3$, we deduce from \eqref{r0} that
\begin{align*}
&|J_1(\vv,\dd)|+|J_3(\vv,\dd)|\\&\le c\int_0^T\Big\{(1+\|\bar
Y_{t_\dd}^\vv\|_\8^{q_1}+\|X_s^0\|_\8^{q_1})\|\bar
Y_{t_\dd}^\vv-X_t^0\|_\8+\mathbb{W}_2(\mathscr{L}_{\bar Y_{t_\dd}^\vv},\mathscr{L}_{X_t^0})\\
&\quad+\dd^{\aa}\Big(1+\|\bar
Y_{s_\dd}^\vv\|_\8^{2(1+q_1)}+\mathbb{W}_2(\mathscr{L}_{\bar Y_{t_\dd}^\vv},\dd_{\zeta_0})^2\Big)\Big\}\\
&\quad\times\Big\{1 +\|X_t^0\|_\8^{1+q_1}+
 \|\bar Y_{t_\dd}^\vv\|_\8^{1+q_1} + \mathbb{W}_2(\mathscr{L}_{\bar
Y_{t_\dd}^\vv},\dd_0)\Big\}\\
&\quad\times\Big\{1 +\|X_t^0\|_\8^{1+q_3}+
 \|\bar Y_{t_\dd}^\vv\|_\8^{1+q_3} + \mathbb{W}_2(\mathscr{L}_{\bar
Y_{t_\dd}^\vv},\dd_{\zeta_0})\Big\}\d t\\
&\le c\int_0^T\Big\{(1+\|\bar Y_{t_\dd}^\vv\|_\8^q)\|\bar
Y_{t_\dd}^\vv-X_t^0\|_\8+\ss{\E\|\bar Y_{t_\dd}^\vv-X_t^0\|_\8^2}\Big\}\\
&\quad\times\Big\{1+
 \|\bar Y_{t_\dd}^\vv\|_\8^{2(1+ q)} + \E\|\bar Y_{t_\dd}^\vv\|_\8^2\Big\}\d s\\
&\quad+c\,\dd^{\aa}\int_0^T\Big\{1+
 \|\bar Y_{t_\dd}^\vv\|_\8^{4(1+ q)} + \E\|\bar
Y_{t_\dd}^\vv\|_\8^4\Big\}\d t.
\end{align*}
This, by exploiting \eqref{r6} and \eqref{a7} and using H\"older's
inequality, gives that
\begin{equation}\label{t5}
\begin{split}
&\E|J_1(\vv,\dd)|+\E|J_3(\vv,\dd)|\\&\le c\,\int_0^T\ss{\E\|\bar
Y_{t_\dd}^\vv-X_t^0\|_\8^2}\Big\{1 +\E
 \|\bar Y_{t_\dd}^\vv\|_\8^{8(1+ q)}\Big\}\d t\\
&\quad+c\,\dd^{\aa}\int_0^T\Big\{1+\E
 \|\bar Y_{t_\dd}^\vv\|_\8^{4(1+ q)} \Big\}\d t\\
&\rightarrow0
\end{split}
\end{equation}
  as  $ \vv\rightarrow0$ and $\dd\rightarrow0$.
Next,  making use of ({\bf B3})   and \eqref{t3}, we derive that
\begin{equation*}
\begin{split}
 |J_2(\vv,\dd)|&\le c\int_0^T(1+
 \|X_t^0\|_\8^{1+q_1})\Big(1+
 \|\bar Y_{t_\dd}^\vv\|_\8^{1+q_1} + \ss{\E\|\bar Y_{t_\dd}^\vv\|_\8^2}\Big)\\
&\quad\times\Big((1+\|\bar
Y_{t_\dd}^\vv\|_\8^{q_3}+\|X_t^0\|_\8^{q_3})\|\bar
Y_{t_\dd}^\vv-X_t^0\|_\8+\ss{\E\|\bar
Y_{t_\dd}^\vv-X_t^0\|_\8^2}\Big)\d t.
\end{split}
\end{equation*}
Again, using \eqref{r0}, \eqref{r6} and \eqref{a9}  and utilizing
H\"older's inequality gives that
\begin{equation}\label{t6}
\begin{split}
\E|J_2(\vv,\dd)|&\le c\,\int_0^T\ss{\E\|\bar
Y_{t_\dd}^\vv-X_t^0\|_\8^2}\Big\{1 +\E
 \|\bar Y_{t_\dd}^\vv\|_\8^{4(1+ q)}\Big\}\d t\\
&\rightarrow0
\end{split}
\end{equation}
 as  $ \vv\rightarrow0$ and  $\dd\rightarrow0$.
Hence,   \eqref{t4} follows immediately from \eqref{t5} and
\eqref{t6}.

In the sequel, we are going to show that \eqref{q6} holds. In terms
of \eqref{q3}, we obtain that
\begin{equation}\label{q5}
\begin{split}
&\sum_{k=1}^n(\Gamma^{(\dd)})^*(\bar Y_{(k-1)\dd}^\vv,\theta,\theta_0)\hat\si(\bar Y_{(k-1)\dd}^\vv)P_k(\theta_0)\\
&=\vv\sum_{k=1}^n(\Gamma^{(\dd)})^*(\bar
Y_{(k-1)\dd}^\vv,\theta,\theta_0)\hat\si(\bar Y_{(k-1)\dd}^\vv)
\si(\bar Y^\vv_{(k-1)\dd},\mathscr{L}_{\bar Y^\vv_{(k-1)\dd}})(B(k\dd)-B((k-1)\dd))\\
&=\vv\int_0^T(\Gamma^{(\dd)})^*(\bar
Y_{t_\dd}^\vv,\theta,\theta_0)\hat\si(\bar Y_{t_\dd}^\vv) \si(\bar
Y^\vv_{t_\dd},\mathscr{L}_{\bar Y^\vv_{t_\dd}})\d B(t).
\end{split}
\end{equation}
By the It\^o isometry and the H\"older  inequality, we derive from
\eqref{r2}, \eqref{t3}, and \eqref{t2} that
\begin{equation*}
\begin{split}
&\E\Big|\int_0^T(\Gamma^{(\dd)})^*(\bar
Y_{t_\dd}^\vv,\theta,\theta_0)\hat\si(\bar Y_{t_\dd}^\vv) \si(\bar
Y^\vv_{t_\dd},\mathscr{L}_{\bar Y^\vv_{t_\dd}})\d
B(t)\Big|^2\\
&=\int_0^T\E|(\Gamma^{(\dd)})^*(\bar
Y_{t_\dd}^\vv,\theta,\theta_0)\hat\si(\bar Y_{t_\dd}^\vv) \si(\bar
Y^\vv_{t_\dd},\mathscr{L}_{\bar Y^\vv_{t_\dd}})|^2\d
 t\\
 &\le\int_0^T\E\{|\Gamma^{(\dd)}(\bar Y_{t_\dd}^\vv,\theta,\theta_0)|^2\cdot\|\hat\si(\bar
Y_{t_\dd}^\vv)\|^2\cdot\| \si(\bar Y^\vv_{t_\dd},\mathscr{L}_{\bar
Y^\vv_{t_\dd}})\|^2\}\d t\\
&\le c\,\int_0^T\E\Big\{\Big(1+\|\bar
Y^\vv_{t_\dd}\|_\8^2+\mathbb{W}_2(\mathscr{L}_{\bar Y^\vv_{t_\dd}},\dd_{\zeta_0})^2\Big)\\
&\quad\times\Big(1+
 \|\bar Y^\vv_{t_\dd}\|_\8^{2(1+q_3)} +
 \mathbb{W}_2(\mathscr{L}_{\bar Y^\vv_{t_\dd}},\dd_{\zeta_0})^2\Big)\\
&\quad\times\Big(1+
 \|\bar Y^\vv_{t_\dd}\|_\8^{2(1+q_1)} +
 \mathbb{W}_2(\mathscr{L}_{\bar Y^\vv_{t_\dd}},\dd_{\zeta_0})^2\Big)\Big\}\d t\\
&\le c\,\int_0^T\{1+\E\|\bar Y^\vv_{t_\dd}\|_\8^{8(1+q)}\}\d t.
\end{split}
\end{equation*}
This, together with \eqref{r6}, leads to
\begin{equation*}
\begin{split}
&\E\Big|\sum_{k=1}^n(\Gamma^{(\dd)})^*(\bar Y_{(k-1)\dd}^\vv,\theta,\theta_0)\hat\si(\bar Y_{(k-1)\dd}^\vv)P_k(\theta_0)\Big|^2\\
&\le c\,\vv^2\int_0^T\{1+\E\|\bar Y^\vv_{t_\dd}\|_\8^{8(1+q)}\}\d t\\
&\le c\,\vv^2.
\end{split}
\end{equation*}
As a consequence, we obtain \eqref{q6} immediately.
\end{proof}

So far, with Lemma \ref{lem} in hand, we are in the position to
complete the
\begin{proof}[ Proof of Theorem \ref{th1}]
A direction calculation shows that
\begin{equation}\label{h1}
\begin{split}
&\Phi_{n,\vv}(\theta)\\
&=\delta^{-1}\sum_{k=1}^n\Big\{P_k^*(\theta)\hat\si(\bar Y_{(k-1)\dd}^\vv)P_k(\theta)-P_k^*(\theta_0)\hat\si(\bar Y_{(k-1)\dd}^\vv)P_k(\theta_0)\Big\}\\
&=\delta^{-1}\sum_{k=1}^n\Big\{\Big(P_k(\theta_0)+(\Gamma^{(\dd)})^*(\bar
Y_{(k-1)\dd}^\vv,\theta,\theta_0)\dd\Big)^*\hat\si(\bar
Y_{(k-1)\dd}^\vv) \Big(P_k(\theta_0)
+\Gamma^{(\dd)}(\bar Y_{t_{k-1}}^\vv,\theta,\theta_0)\dd\Big)\\
&\quad -P_k^*(\theta_0)\hat\si(\bar Y_{(k-1)\dd}^\vv)P_k(\theta_0)\Big\}\\
&=2\sum_{k=1}^n(\Gamma^{(\dd)})^*(\bar Y_{(k-1)\dd}^\vv,\theta,\theta_0)\hat\si(\bar Y_{(k-1)\dd}^\vv)P_k(\theta_0)\\
&\quad+\dd\sum_{k=1}^n(\Gamma^{(\dd)})^*(\bar
Y_{(k-1)\dd}^\vv,\theta,\theta_0)\hat\si(\bar
Y_{(k-1)\dd}^\vv)\Gamma^{(\dd)}(\bar
Y_{(k-1)\dd}^\vv,\theta,\theta_0).
\end{split}
\end{equation}
By virtue of  Lemma \ref{lem}, we therefore infer from Chebyshev's
inequality that
\begin{equation*}
\sup_{\theta\in\Theta}|-\Phi_{n,\vv}(\theta)-(-\Xi(\theta))|\rightarrow0~~~~\mbox{
in probability.}
\end{equation*}
 Next, for any
$\kk>0,$ due to $\Xi(\cdot)>0$,
\begin{equation*}
 \sup_{|\theta-\theta_0|\ge\kk}(-\Xi(\theta))<-\Xi(\theta_0)=0.
\end{equation*}
Furthermore,  one has $-\Phi_{n,\vv}(
\hat\theta_{n,\vv})\ge-\Phi_{n,\vv}(\theta_0)=0$.  Consequently, we
deduce  from \cite[Theorem 5.9]{V98}  with
$M_n(\cdot)=-\Phi_{n,\vv}(\cdot)$ and $M(\cdot)=-\Xi(\cdot)$ therein
that $\hat\theta_{n,\vv}\rightarrow\theta_0$ in probability as
$\vv\rightarrow0$ and $n\rightarrow\8$. We henceforth complete the
proof.
\end{proof}

\section{Proof of Theorem \ref{th2}}\label{sec3}
Before we start to finish the argument of Theorem \ref{th2}, we also
need to  prepare some auxiliary lemmas below. For any random
variable $\zeta\in\C$ with $\mathscr{L}_\zeta\in\mathcal {P}_2(\C)$,
set
\begin{equation*}
\Upsilon^{(\dd)}(\zeta,\theta):=(\nn_\theta b^{(\dd)})^*
(\zeta,\mathscr{L}_\zeta,\theta)\hat\si(\zeta)\si(\zeta,\mathscr{L}_{\zeta}).
\end{equation*}

\begin{lem}\label{le2}
 Let $({\bf A1})-({\bf A2})$ and $({\bf B1})-({\bf B4})$ hold. Then,
\begin{equation}\label{v1}
\vv^{-1}(\nn_\theta\Phi_{n,\vv})(\theta)\rightarrow-2\int_0^T\Upsilon(X_t^0,\theta)\d
B(t)~~~~\mbox{ in probability }
\end{equation}
whenever $\vv\rightarrow0$ and $n\rightarrow\8$, where
$\Upsilon(\cdot,\cdot)$ is introduced in \eqref{0s0}.
\end{lem}

\begin{proof}
By the chain rule,   one infers from \eqref{q3} and \eqref{h1} that
\begin{equation}\label{c1}
\begin{split}
&\vv^{-1}(\nn_\theta\Phi_{n,\vv})(\theta)\\%&=2\vv^{-1}\sum_{k=1}^n(\nn_\theta\Gamma^{(\dd)})^*(\bar Y_{(k-1)\dd}^\vv,\theta,\theta_0)\hat\si(\bar Y_{(k-1)\dd}^\vv)P_k(\theta_0)\\
%&\quad+2\vv^{-1}\dd\sum_{k=1}^n(\nn_\theta\Gamma^{(\dd)})^*(\bar
%Y_{(k-1)\dd}^\vv,\theta,\theta_0)\hat\si(\bar
%Y_{(k-1)\dd}^\vv)\Gamma^{(\dd)}(\bar
%Y_{(k-1)\dd}^\vv,\theta,\theta_0)\\
&=2\,\vv^{-1}\sum_{k=1}^n(\nn_\theta\Gamma^{(\dd)})^*(\bar
Y_{(k-1)\dd}^\vv,\theta,\theta_0)\hat\si(\bar Y_{(k-1)\dd}^\vv)
\Big\{P_k(\theta_0)+\Gamma^{(\dd)}(\bar
Y_{(k-1)\dd}^\vv,\theta,\theta_0)\dd\Big\}\\
&=2\,\vv^{-1}\sum_{k=1}^n(\nn_\theta\Gamma^{(\dd)})^*(\bar Y_{(k-1)\dd}^\vv,\theta,\theta_0)\hat\si(\bar Y_{(k-1)\dd}^\vv)P_k(\theta)\\
&=-2\sum_{k=1}^n(\nn_\theta b^{(\dd)})^*(\bar
Y_{(k-1)\dd}^\vv,\mathscr{L}_{\bar
Y_{(k-1)\dd}^\vv},\theta)\hat\si(\bar Y_{(k-1)\dd}^\vv)\si(\bar
Y_{(k-1)\dd}^\vv,\mathscr{L}_{\bar
Y_{(k-1)\dd}^\vv})\\
&\quad\times(B(k\dd)-B((k-1)\dd))\\
&=-2\int_0^T\Upsilon^{(\dd)}(\bar Y_{t_\dd}^\vv,\theta)\d B(t),
\end{split}
\end{equation}
where in the last two display we used the fact that
\begin{equation}\label{c2}
(\nn_\theta\Gamma^{(\dd)})(\bar
Y_{(k-1)\dd}^\vv,\theta,\theta_0)=-(\nn_\theta b^{(\dd)}) (\bar
Y_{(k-1)\dd}^\vv,\mathscr{L}_{\bar Y_{(k-1)\dd}^\vv},\theta).
\end{equation}
To achieve \eqref{v1}, in terms of \cite[Theorem 2.6, P.63]{F98}, it
is sufficient to claim that
\begin{equation}\label{v2}
\int_0^T\|\Upsilon^{(\dd)}(\bar
Y_{t_\dd}^\vv,\theta)-\Upsilon(X_t^0,\theta)\|^2\d
t\rightarrow0~~~~\mbox{ in probability }
\end{equation}
as $\vv\rightarrow0$ and $\dd\rightarrow0.$ Observe that
\begin{equation*}
\begin{split}
&\Upsilon^{(\dd)}(\bar
Y_{t_\dd}^\vv,\theta)-\Upsilon(X_t^0,\theta)\\
&=(\nn_\theta b^{(\dd)})^* (\bar Y_{t_\dd}^\vv,\mathscr{L}_{\bar
Y_{t_\dd}^\vv},\theta)\hat\si(\bar Y_{t_\dd}^\vv)\si(\bar
Y_{t_\dd}^\vv,\mathscr{L}_{\bar Y_{t_\dd}^\vv})-(\nn_\theta b)^*
(X_t^0,\mathscr{L}_{X_t^0},\theta)\hat\si(X_t^0)\si(X_t^0,\mathscr{L}_{X_t^0})\\
&=\{(\nn_\theta b^{(\dd)})^* (\bar Y_{t_\dd}^\vv,\mathscr{L}_{\bar
Y_{t_\dd}^\vv},\theta)-(\nn_\theta b)^*
(X_t^0,\mathscr{L}_{X_t^0},\theta)\}\hat\si(\bar Y_{t_\dd}^\vv)\si(\bar Y_{t_\dd}^\vv,\mathscr{L}_{\bar Y_{t_\dd}^\vv})\\
&\quad+(\nn_\theta b)^*
(X_t^0,\mathscr{L}_{X_t^0},\theta)\{\hat\si(\bar
Y_{t_\dd}^\vv)-\hat\si(X_t^0)\}\si(\bar Y_{t_\dd}^\vv,
\mathscr{L}_{\bar Y_{t_\dd}^\vv})\\
&\quad+(\nn_\theta b)^* (X_t^0,\mathscr{L}_{X_t^0},\theta)
\hat\si(X_t^0)\{\si(\bar Y_{t_\dd}^\vv,\mathscr{L}_{\bar
Y_{t_\dd}^\vv})-\si(X_t^0,\mathscr{L}_{X_t^0})\}\\
&=:\Sigma_1(t,\vv,\dd)+\Sigma_2(t,\vv,\dd)+\Sigma_3(t,\vv,\dd).
\end{split}
\end{equation*}

By a straightforward calculation, for any random variable
$\zeta\in\C$ with $\mathscr{L}_{\zeta}\in\mathcal {P}_2(\C)$, one
has
\begin{equation}\label{v3}
\begin{split}
(\nn_\theta
b^{(\dd)})(\zeta,\mathscr{L}_\zeta,\theta)&=\nn_\theta\Big(\ff{b(\zeta,\mu,\theta)}{1+\dd^\aa|b(\zeta,\mu,\theta)|}\Big)\\
&=\ff{(\nn_\theta
b)(\zeta,\mu,\theta)}{1+\dd^\aa|b(\zeta,\mu,\theta)|}-\ff{\dd^\aa( b
b^*)(\zeta,\mu,\theta)(\nn_\theta
b)(\zeta,\mu,\theta)}{|b(\zeta,\mu,\theta)|(1+\dd^\aa|b(\zeta,\mu,\theta)|)^2}.
\end{split}
\end{equation}
Next, for   any random variables $\zeta_1,\zeta_2\in\C$ with
$\mathscr{L}_{\zeta_1},\mathscr{L}_{\zeta_2}\in\mathcal {P}_2(\C)$,
it follows from \eqref{v3} that
\begin{equation}\label{v5}
\begin{split}
&\|(\nn_\theta b^{(\dd)})^*
(\zeta_1,\mathscr{L}_{\zeta_1},\theta)-(\nn_\theta b)^*
(\zeta_2,\mathscr{L}_{\zeta_2},\theta)\|\\
&=\Big\|\ff{(\nn_\theta
b)^*(\zeta_1,\mathscr{L}_{\zeta_1},\theta)}{1+\dd^\aa|b(\zeta_1,\mathscr{L}_{\zeta_1},\theta)|}-(\nn_\theta
b)^* (\zeta_2,\mathscr{L}_{\zeta_2},\theta) -\ff{\dd^\aa (\nn_\theta
b)^*(\zeta_1,\mathscr{L}_{\zeta_1},\theta) (b
b^*)(\zeta_1,\mathscr{L}_{\zeta_1},\theta)}{(1+\dd^\aa|b(\zeta_1,\mathscr{L}_{\zeta_1},\theta)|)^2|b(\zeta_1,\mathscr{L}_{\zeta_1},\theta)|}\Big\|\\
&=\Big\|\ff{(\nn_\theta
b)^*(\zeta_1,\mathscr{L}_{\zeta_1},\theta)-(\nn_\theta b)^*
(\zeta_2,\mathscr{L}_{\zeta_2},\theta)}{1+\dd^\aa|b(\zeta_1,\mathscr{L}_{\zeta_1},\theta)|}-\ff{\dd^\aa|b(\zeta_1,\mathscr{L}_{\zeta_1},\theta)|
(\nn_\theta b)^*
(\zeta_2,\mathscr{L}_{\zeta_2},\theta)}{1+\dd^\aa|b(\zeta_1,\mathscr{L}_{\zeta_1},\theta)|}\\
&\quad-\ff{\dd^\aa (\nn_\theta
b)^*(\zeta_1,\mathscr{L}_{\zeta_1},\theta) (b
b^*)(\zeta_1,\mathscr{L}_{\zeta_1},\theta)}{(1+\dd^\aa|b(\zeta_1,\mathscr{L}_{\zeta_1},\theta)|)^2|b(\zeta_1,\mathscr{L}_{\zeta_1},\theta)|}\Big\|\\
&\le\|(\nn_\theta
b)(\zeta_1,\mathscr{L}_{\zeta_1},\theta)-(\nn_\theta b)
(\zeta_2,\mathscr{L}_{\zeta_2},\theta)\|\\
&\quad+\dd^\aa|b(\zeta_1,\mathscr{L}_{\zeta_1},\theta)|\cdot\{\|(\nn_\theta
b) (\zeta_2,\mathscr{L}_{\zeta_2},\theta)\|+\|(\nn_\theta b)
(\zeta_1,\mathscr{L}_{\zeta_1},\theta)\|\},
\end{split}
\end{equation}
where in the last step we utilized the facts that $\|A\|=\|A^*\|$
for a matrix $A$ and that
\begin{equation*}
\begin{split}
&\|(\nn_\theta b)^*(\zeta_1,\mathscr{L}_{\zeta_1},\theta) (b
b^*)(\zeta_1,\mathscr{L}_{\zeta_1},\theta)\|^2\\
&=\mbox{trace}\Big(((\nn_\theta
b)^*(\zeta_1,\mathscr{L}_{\zeta_1},\theta) (b
b^*)(\zeta_1,\mathscr{L}_{\zeta_1},\theta))^*(\nn_\theta
b)^*(\zeta_1,\mathscr{L}_{\zeta_1},\theta) (b
b^*)(\zeta_1,\mathscr{L}_{\zeta_1},\theta)\Big)\\
&=\mbox{trace}\Big((b
b^*)^*(\zeta_1,\mathscr{L}_{\zeta_1},\theta))((\nn_\theta b)
(\nn_\theta b)^*)(\zeta_1,\mathscr{L}_{\zeta_1},\theta) (b
b^*)(\zeta_1,\mathscr{L}_{\zeta_1},\theta)\Big)\\
&=\mbox{trace}\Big(((\nn_\theta b) (\nn_\theta
b)^*)(\zeta_1,\mathscr{L}_{\zeta_1},\theta)(b
b^*)(\zeta_1,\mathscr{L}_{\zeta_1},\theta)(b b^*)^*(\zeta_1,\mathscr{L}_{\zeta_1},\theta)) \Big)\\
&=|b(\zeta_1,\mathscr{L}_{\zeta_1},\theta)|^4\|(\nn_\theta b)^*
(\zeta_1,\mathscr{L}_{\zeta_1},\theta)\|^2.
\end{split}
\end{equation*}
Moreover, from  ({\bf B2}), one has
\begin{equation}\label{v4}
\|(\nn_\theta b) (\zeta_2,\mathscr{L}_{\zeta_2},\theta)\|\le
c\Big\{1+\|\zeta_2\|_\8^{1+q_2}+\mathbb{W}_2(\mathscr{L}_{\zeta_2},\dd_{\zeta_0})\Big\}.
\end{equation}
Now, taking ({\bf B2}), \eqref{v5}, and \eqref{v4}, in addition to
\eqref{r2} and \eqref{t2},  into account yields that
\begin{equation*}
\begin{split}
\|\Sigma_1(t,\vv,\dd)\|&\le c\Big\{(1+\|\bar
Y_{t_\dd}^\vv\|_\8^{q_2}+\|X_t^0\|_\8^{q_2})\|\bar
Y_{t_\dd}^\vv-X_t^0\|_\8+\mathbb{W}_2(\mathscr{L}_{\bar
Y_{t_\dd}^\vv},\mathscr{L}_{X_t^0})\\
&\quad+\dd^\aa\Big(1+\|\bar
Y_{t_\dd}^\vv\|_\8^{1+q_1}+\mathbb{W}_2(\mathscr{L}_{\bar
Y_{t_\dd}^\vv},\dd_{\zeta_0})\Big)\Big(1+\|\bar
Y_{t_\dd}^\vv\|_\8^{1+q_2}+\mathbb{W}_2(\mathscr{L}_{\bar
Y_{t_\dd}^\vv},\dd_{\zeta_0})\Big)\Big\}\\
&\quad\times\Big\{1+
 \|\bar
Y_{t_\dd}^\vv\|_\8^{1+q_3} +
 \mathbb{W}_2(\mathscr{L}_{\bar
Y_{t_\dd}^\vv},\dd_{\zeta_0})\Big\}\times\Big\{1+\|\bar
Y_{t_\dd}^\vv\|_\8+\mathbb{W}_2(\mathscr{L}_{\bar
Y_{t_\dd}^\vv},\dd_{\zeta_0})\Big\}.
\end{split}
\end{equation*}
For $q:=q_1\vee q_2\vee q_3,$ simple calculations and \eqref{r6}
 give  that
\begin{equation*}
\begin{split}
\|\Sigma_1(t,\vv,\dd)\|&\le c\Big\{(1+\|\bar
Y_{t_\dd}^\vv\|_\8^{q})\|\bar
Y_{t_\dd}^\vv-X_t^0\|_\8+\ss{\E\|\mathscr{L}_{\bar
Y_{t_\dd}^\vv}-X_t^0\|_\8^2}\Big\}\\
&\quad\times\Big\{1+\|\bar Y_{t_\dd}^\vv\|_\8^{2(1+q)}+\E\|\bar
Y_{t_\dd}^\vv\|_\8^2\Big\}\\
&\quad+c\dd^\aa\Big\{1+\|\bar Y_{t_\dd}^\vv\|_\8^{2(1+q)}+\E\|\bar
Y_{t_\dd}^\vv\|_\8^2\Big\}^2\\
&\le c (1+\|\bar Y_{t_\dd}^\vv\|_\8^{4(1+q)})\|\bar
Y_{t_\dd}^\vv-X_t^0\|_\8\\
&\quad+c (1+\|\bar
Y_{t_\dd}^\vv\|_\8^{2(1+q)})\ss{\E\|\mathscr{L}_{\bar
Y_{t_\dd}^\vv}-X_t^0\|_\8^2}+c\dd^\aa(1+\|\bar
Y_{t_\dd}^\vv\|_\8^{4(1+q)})\\
&=:\tt\LL_1(t,\vv,\dd)+\tt\LL_2(t,\vv,\dd)+\tt\LL_3(t,\vv,\dd).
\end{split}
\end{equation*}
For any $\rho>0$, by virtue of H\"older's inequality, together with
\eqref{r6} and \eqref{a9}, it follows that
\begin{equation}\label{v6}
\begin{split}
&\P\Big(\int_0^T\|\tt\LL_1(t,\vv,\dd)\|^2\d t\ge\rho\Big)\\
&\le\P\Big(c\int_0^T (1+\|\bar Y_{t_\dd}^\vv\|_\8^{8(1+q)})\|\bar
Y_{t_\dd}^\vv-X_t^0\|_\8^2\d t\ge\rho\Big)\\
&\le\P\Big(c\int_0^T (1+\|\bar Y_{t_\dd}^\vv\|_\8^{9(1+q)})\|\bar
Y_{t_\dd}^\vv-X_t^0\|_\8\d t\ge\rho\Big)\\
&\le \ff{c}{\rho}\int_0^T  (1+\E\|\bar
Y_{t_\dd}^\vv\|_\8^{18(1+q)})\ss{\E\|\bar
Y_{t_\dd}^\vv-X_t^0\|_\8^2} \d t\\
&\rightarrow0
\end{split}
\end{equation}
whenever $\vv\rightarrow0$ and $\dd\rightarrow0.$ On the other hand,
by means of \eqref{r6}, and \eqref{a9}, it follows that
\begin{equation}\label{v7}
\begin{split}
\E\tt\LL_2^2(t,\vv,\dd)+\E\tt\LL_3^2(t,\vv,\dd)&\le c (1+\E\|\bar
Y_{t_\dd}^\vv\|_\8^{4(1+q)})\E\|\mathscr{L}_{\bar
Y_{t_\dd}^\vv}-X_t^0\|_\8^2+c\dd^\aa(1+\E\|\bar
Y_{t_\dd}^\vv\|_\8^{8(1+q)})\\
&\le c(\dd^\bb+\vv^2+\dd^{\aa})\\
&\rightarrow0
\end{split}
\end{equation}
as $\vv\rightarrow0$ and $\dd\rightarrow0.$ As a consequence, we
infer from \eqref{v6} and \eqref{v7} that
\begin{equation}\label{b1}
 \int_0^T \|\Sigma_1(t,\vv,\dd) \|^2\d t\rightarrow0~~~~\mbox{ in
 probability }
\end{equation}
when $\vv\rightarrow0$ and $\dd\rightarrow0.$ Next, taking advantage
of ({\bf A2}), ({\bf B3}), \eqref{r2}, and \eqref{v4} leads to
\begin{equation*}
\begin{split}
&\|\Sigma_2(t,\vv,\dd)\|^2+\|\Sigma_3(t,\vv,\dd)\|^2\\&\le c\Big\{1+\|X_t^0\|_\8^{2(1+q)}\Big\}\\
&\quad\times\Big\{(1+\|\bar
Y_{t_\dd}^\vv\|_\8^{2q}+\|X_t^0\|_\8^{2q})\|\bar
Y_{t_\dd}^\vv-X_t^0\|_\8^2+\mathbb{W}_2(\mathscr{L}_{\bar
Y_{t_\dd}^\vv},\mathscr{L}_{X_t^0})^2\Big\}\\
&\quad\times (1+\|\bar
Y_{t_\dd}^\vv\|_\8^2+\mathbb{W}_2(\mathscr{L}_{\bar
Y_{t_\dd}^\vv},\dd_{\zeta_0})^2)\\
&\le  c\Big\{(1+\|\bar Y_{t_\dd}^\vv\|_\8^{2(1+q)}) \|\bar
Y_{t_\dd}^\vv-X_t^0\|_\8+\E\|\bar Y_{t_\dd}^\vv-X_t^0\|^2_\8\Big\}\\
&\quad\times  \Big\{1+\|\bar Y_{t_\dd}^\vv\|_\8^2+\E\|\bar
Y_{t_\dd}^\vv\|^2_\8\Big\}\\
&\le c (1+\|\bar Y_{t_\dd}^\vv\|_\8^{2(2+q)}) \|\bar
Y_{t_\dd}^\vv-X_t^0\|_\8 +c(1+\|\bar Y_{t_\dd}^\vv\|_\8^2)\E\|\bar
Y_{t_\dd}^\vv-X_t^0\|^2_\8\\
&=:\Xi_1(t,\vv,\dd)+\Xi_2(t,\vv,\dd),
\end{split}
\end{equation*}
in which we adopted \eqref{r6} in the last procedure. Via H\"older's
inequality, we obtain from \eqref{r6} and \eqref{a9} that
\begin{equation}\label{v8}
 \E \Xi_1(t,\vv,\dd)\le c(1+\E\|\bar Y_{t_\dd}^\vv\|_\8^{4(2+q)}) \ss{\E\|\bar
Y_{t_\dd}^\vv-X_t^0\|_\8^2}\rightarrow0
\end{equation}
as $\vv\rightarrow0$ and $\dd\rightarrow0.$ Also, by \eqref{r6} and
\eqref{a9}, one has
\begin{equation}\label{v9}
 \E \Xi_2(t,\vv,\dd)\le c(1+\E\|\bar Y_{t_\dd}^\vv\|_\8^2)  \E\|\bar
Y_{t_\dd}^\vv-X_t^0\|_\8^2\rightarrow0
\end{equation}
provided that  $\vv\rightarrow0$ and $\dd\rightarrow0.$ Therefore,
\eqref{v8} and \eqref{v9} lead to
\begin{equation}\label{b2}
 \E \|\Sigma_2(t,\vv,\dd)\|^2+  \E \|\Sigma_3(t,\vv,\dd)\|^2\rightarrow0
\end{equation}
if  $\vv\rightarrow0$ and $\dd\rightarrow0.$ At last, the desired
assertion \eqref{v1} holds from \eqref{b1} and \eqref{b2}.
\end{proof}

\begin{lem}\label{le3}
  Let $({\bf A1})-({\bf A3}), ({\bf B1})-({\bf B4})$, and $({\bf
  C})$
hold. Then
\begin{equation}\label{c3}
(\nn_\theta^{(2)}\Phi_{n,\vv})(\theta)\rightarrow
K_0(\theta):=K(\theta)+I(\theta)~~~~\mbox{ in probability }
\end{equation}
as $n\rightarrow\8$ and $\vv\rightarrow0$, where $I(\cdot)$ and
$K(\cdot)$ are introduced in \eqref{0z3} and \eqref{0z2},
respectively.
\end{lem}

\begin{proof}
 From \eqref{c1} and \eqref{c2}, we deduce that
\begin{equation*}
\begin{split}
(\nn_\theta^{(2)}\Phi_{n,\vv})(\theta)& =2
\sum_{k=1}^n(\nn_\theta^{(2)}(\Gamma^{(\dd)})^*)(\bar
Y_{(k-1)\dd}^\vv,\theta,\theta_0)
\circ\Big(\hat\si(\bar Y_{(k-1)\dd}^\vv)P_k(\theta)\Big)\\
&\quad+2 \sum_{k=1}^n(\nn_\theta\Gamma^{(\dd)})^*(\bar
Y_{(k-1)\dd}^\vv,\theta,\theta_0)\hat\si(\bar
Y_{(k-1)\dd}^\vv)(\nn_\theta P_k)(\theta)\\
& =-2 \sum_{k=1}^n(\nn_\theta^{(2)}(b^{(\dd)})^*)(\bar
Y_{(k-1)\dd}^\vv,\mathscr{L}_{\bar Y_{(k-1)\dd}^\vv},\theta)
\circ\Big(\hat\si(\bar Y_{(k-1)\dd}^\vv)P_k(\theta)\Big)\\
&\quad+2 \dd\sum_{k=1}^n(\nn_\theta b^{(\dd)})^* (\bar
Y_{(k-1)\dd}^\vv,\mathscr{L}_{\bar
Y_{(k-1)\dd}^\vv},\theta)\hat\si(\bar Y_{(k-1)\dd}^\vv)(\nn_\theta
b^{(\dd)})(\bar Y_{(k-1)\dd}^\vv,\mathscr{L}_{\bar
Y_{(k-1)\dd}^\vv},\theta).
\end{split}
\end{equation*}
For any random variable $\zeta\in\C$ with
$\mathscr{L}_\zeta\in\mathcal {P}_2(\C)$, by the chain rule, we
infer from \eqref{v3} that
\begin{equation}\label{e6}
\begin{split}
\Big(\nn_\theta^{(2)}(b^{(\dd)})^*\Big)(\zeta,\mathscr{L}_\zeta,\theta)
&=\bigg(\nn_\theta\bigg(\ff{(\nn_\theta
b^*)}{1+\dd^\aa|b|}\bigg)\bigg)(\zeta,\mathscr{L}_\zeta,\theta)\\
&\quad-\dd^\aa\bigg(\nn_\theta\bigg(\ff{(\nn_\theta b^*)( b
b^*)}{|b|(1+\dd^\aa|b|)^2}\bigg)\bigg)(\zeta,\mathscr{L}_\zeta,\theta)\\
&=(\nn_\theta^{(2)} b^*)
(\zeta,\mathscr{L}_\zeta,\theta)-\dd^\aa\Theta_1(\zeta,\mathscr{L}_\zeta,\theta).
\end{split}
\end{equation}
Next, the chain rule shows that
\begin{equation*}
\begin{split}
& \Theta_1(\zeta,\mathscr{L}_\zeta,\theta)
:= \bigg(\ff{|b|(\nn_\theta^{(2)} b^*)}{1+\dd^\aa|b|}+\ff{\Big(b^*(\ff{\partial}{\partial\theta_1}b)(\nn_\theta
b)^*,\cdots,b^*(\ff{\partial}{\partial\theta_p}b)(\nn_\theta
b)^*\Big)_{p\times pd}}{|b|(1+\dd^\aa|b|)^2}\\
&\qquad\qquad\quad+\ff{\Big((\ff{\partial}{\partial\theta_1}(\nn_\theta b^*))( b
b^*),\cdots,(\ff{\partial}{\partial\theta_p}(\nn_\theta b)^*)( b
b^*)\Big)_{p\times pd}}{|b|(1+\dd^\aa|b|)^2}\\
&\qquad\qquad\quad+\ff{\Big((\nn_\theta b)^*( (\ff{\partial}{\partial\theta_1}b)
b^*+b\ff{\partial}{\partial\theta_1}b^*),\cdots,(\nn_\theta b)^*(
(\ff{\partial}{\partial\theta_p}b)
b^*+b\ff{\partial}{\partial\theta_p}b^*)\Big)_{p\times pd}}{|b|(1+\dd^\aa|b|)^2}\\
&\quad-\ff{1+3\dd^\aa|b|}{|b|^3(1+\dd^\aa|b|)^3}\Big(
\Big(b^*\Big(\ff{\partial}{\partial\theta_1}b\Big)\Big)(\nn_\theta
b)^*( b
b^*),\cdots,\Big(b^*\Big(\ff{\partial}{\partial\theta_p}b\Big)\Big)(\nn_\theta
b)^*( b b^*)\Big)\bigg)_{p\times
pd}(\zeta,\mathscr{L}_\zeta,\theta).
\end{split}
\end{equation*}
 Thanks to \eqref{v3}, it follows
that
\begin{equation}\label{e5}
\begin{split}
&\Big((\nn_\theta b^{(\dd)})^*\hat\si(\zeta)(\nn_\theta
b^{(\dd)})\Big)(\zeta,\mathscr{L}_\zeta,\theta) = \Big((\nn_\theta
b^*)\hat\si(\zeta)(\nn_\theta
b)\Big)(\zeta,\mathscr{L}_\zeta,\theta)-  \dd^{
\aa}\Theta_2(\zeta,\mathscr{L}_\zeta,\theta),
\end{split}
\end{equation}
where
\begin{equation*}
\begin{split}
\Theta_2(\zeta,\mathscr{L}_\zeta,\theta):&=\bigg( \ff{ (2|
b|+\dd^{\aa}| b |^2)(\nn_\theta b^*) \hat\si(\zeta)(\nn_\theta b)
}{(1+\dd^\aa| b |)^2}  +\ff{(\nn_\theta b^*) \hat\si(\zeta)(   b
\,\, b^*) (\nn_\theta
  b) }{|
b |(1+\dd^\aa|
b |)^3} \\
&\quad+ \ff{(\nn_\theta b^*) (  b \, b^*) \hat\si(\zeta)(\nn_\theta
b) }{| b |(1+\dd^\aa| b |)^3} - \dd^{ \aa} \ff{(\nn_\theta b^*) (  b
\, b^*) \hat\si(\zeta)( b \, b^*) (\nn_\theta b) }{| b
|^2(1+\dd^\aa| b |)^4}\bigg)(\zeta,\mathscr{L}_\zeta,\theta).
\end{split}
\end{equation*}
Thus, taking \eqref{e5} and \eqref{e6} into consideration yields
that
\begin{equation*}
\begin{split}
(\nn_\theta^{(2)}\Phi_{n,\vv})(\theta)& =  -2
\dd\sum_{k=1}^n(\nn_\theta^{(2)}b^*)(\bar
Y_{(k-1)\dd}^\vv,\mathscr{L}_{\bar Y_{(k-1)\dd}^\vv},\theta)
\circ\Big(\hat\si(\bar Y_{(k-1)\dd}^\vv)\Gamma^{(\dd)}(\bar
Y_{(k-1)\dd}^\vv,\theta,\theta_0)\Big)\\
&\quad+2 \dd\sum_{k=1}^n\Big((\nn_\theta
b^*)\hat\si(\zeta)(\nn_\theta b)\Big)(\bar
Y_{(k-1)\dd}^\vv,\mathscr{L}_{\bar Y_{(k-1)\dd}^\vv},\theta)\\
&\quad-2 \sum_{k=1}^n(\nn_\theta^{(2)}b^*)(\bar
Y_{(k-1)\dd}^\vv,\mathscr{L}_{\bar Y_{(k-1)\dd}^\vv},\theta)
\circ\Big(\hat\si(\bar Y_{(k-1)\dd}^\vv)P_k(\theta_0)\Big)\\
&\quad-2\dd^\aa \sum_{k=1}^n\Theta_1(\bar
Y_{(k-1)\dd}^\vv,\mathscr{L}_{\bar Y_{(k-1)\dd}^\vv},\theta)
\circ\Big(\hat\si(\bar Y_{(k-1)\dd}^\vv)P_k(\theta)\Big)\\
&\quad- \dd^{ 1+\aa}\sum_{k=1}^n\Theta_2(\bar
Y_{(k-1)\dd}^\vv,\mathscr{L}_{\bar Y_{(k-1)\dd}^\vv},\theta)\\
&=:\sum_{i=1}^5I_i(n,\vv).
\end{split}
\end{equation*}
By following the argument to derive \eqref{t4},   we deduce from
({\bf A3}) that
\begin{equation}\label{c6}
I_1(n,\vv)\rightarrow K(\theta) ~~\mbox{ and
}~~I_2(n,\vv)\rightarrow I(\theta),~~~~\mbox{ in probability }
\end{equation}
as $\vv\rightarrow0$ and $\dd\rightarrow0$. Notice from ({\bf A3})
and \eqref{q4} that
\begin{equation}\label{c4}
\begin{split}
\|\Theta_1\|(\zeta,\mathscr{L}_\zeta,\theta)&\le c\Big((|b|+
\|\nn_\theta^{(2)} b\|+(1+3 |b|)\|\nn_\theta^{(2)}
b\|)\|\nn_\theta^{(2)} b\|\Big)(\zeta,\mathscr{L}_\zeta,\theta)\\
&\le c(1+\|\zeta\|_\8^{4(1+q)}+\mathcal
{W}_2(\mathscr{L}_\zeta,\dd_{\zeta_0})^4).
\end{split}
\end{equation}
On the other hand, owing to \eqref{v4}, \eqref{t2}, and \eqref{q4},
one has
\begin{equation}\label{c5}
\begin{split}
\|\Theta_2\|(\zeta,\mathscr{L}_\zeta,\theta)&\le 2\Big(|
b|~\|\nn_\theta b\|^2\| \hat\si(\zeta)\| (1+ 2| b | )
\Big)(\zeta,\mathscr{L}_\zeta,\theta)\\
&\le c(1+\|\zeta\|_\8^{4(1+q)}+\mathcal
{W}_2(\mathscr{L}_\zeta,\dd_{\zeta_0})^4).
\end{split}
\end{equation}
Thus, by mimicking the argument of \eqref{q6}, we obtain from
\eqref{c4} that
\begin{equation}\label{c7}
I_3(n,\vv)\rightarrow0~~\mbox{ and
}~~I_4(n,\vv)\rightarrow0~~~\mbox{ in probability }
\end{equation}
as $\vv\rightarrow0$ and $\dd\rightarrow0$. Furthermore, \eqref{r6}
and \eqref{c5} enable us to get that
\begin{equation}\label{c8}
I_5(n,\vv)\rightarrow0 ~~~\mbox{ in probability }
\end{equation}
whenever  $\vv\rightarrow0$ and $\dd\rightarrow0$.
 Thus, the desired
assertion \eqref{c3} follows from \eqref{c6}, \eqref{c7}, as well as
\eqref{c8}.
\end{proof}

Now, we move forward to complete the
\begin{proof}[ Proof of Theorem \ref{th2}]
With Lemmas \ref{le2} and \ref{le3} at hand, the proof of Theorem
\ref{th2} is parallel to that of \cite[Theorem 4.1]{RW}. Whereas, to
make the content self-contained, we give an outline of the proof. In
terms of Theorem \ref{th1}, there exists a sequence
$\eta_{n,\vv}\rightarrow0$ as $\vv\rightarrow0$ and $n\rightarrow\8$
such that $\hat\theta_{n,\vv}\in
B_{\eta_{n,\vv}}(\theta_0)\subset\Theta$, $\P$-a.s.  By the Taylor
expansion, one has
\begin{equation}\label{a4}
(\nn_\theta\Phi_{n,\vv})(\hat\theta_{n,\vv})=(\nn_\theta\Phi_{n,\vv})(\theta_0)+D_{n,\vv}(\hat\theta_{n,\vv}-\theta_0),~~~\hat\theta_{n,\vv}\in
B_{\eta_{n,\vv}}(\theta_0)
\end{equation}
with
\begin{equation*}
D_{n,\vv}:=\int_0^1(\nn_\theta^{(2)}\Phi_{n,\vv})(\theta_0
+u(\hat\theta_{n,\vv}-\theta_0))\d u,~~~~\hat\theta_{n,\vv}\in
B_{\eta_{n,\vv}}(\theta_0).
\end{equation*}
 Observe that, for $\hat\theta_{n,\vv}\in
B_{\eta_{n,\vv}}(\theta_0)$,
\begin{equation*}
\begin{split}
\|D_{n,\vv}-K_0(\theta_0)\|&\le\|D_{n,\vv}-(\nn_\theta^{(2)}\Phi_{n,\vv})(\theta_0)\|+\|(\nn_\theta^{(2)}\Phi_{n,\vv})(\theta_0)-K_0(\theta_0)\|\\
&\le\int_0^1\|(\nn_\theta^{(2)}\Phi_{n,\vv})(\theta_0
+u(\hat\theta_{n,\vv}-\theta_0))-(\nn_\theta^{(2)}\Phi_{n,\vv})(\theta_0)\|\d
u\\
&\quad+ \|(\nn_\theta^{(2)}\Phi_{n,\vv})(\theta_0)-K_0(\theta_0)\|\\
&\le \sup_{\theta\in
B_{\eta_{n,\vv}}(\theta_0)}\|(\nn_\theta^{(2)}\Phi_{n,\vv})(\theta)-(\nn_\theta^{(2)}\Phi_{n,\vv})(\theta_0)\|+
\|(\nn_\theta^{(2)}\Phi_{n,\vv})(\theta_0)-K_0(\theta_0)\|\\
&\le\sup_{\theta\in
B_{\eta_{n,\vv}}(\theta_0)}\|(\nn_\theta^{(2)}\Phi_{n,\vv})(\theta)-K_0(\theta)\|+\sup_{\theta\in
B_{\eta_{n,\vv}}(\theta_0)}\|K_0(\theta)-K_0(\theta_0)\|\\
&\quad+2\|(\nn_\theta^{(2)}\Phi_{n,\vv})(\theta_0)-K_0(\theta_0)\|.
\end{split}
\end{equation*}
 This, together
with Lemma \ref{le3} and  continuity of $K_0(\cdot)$, gives that
\begin{equation}\label{a0}
D_{n,\vv}\rightarrow K_0(\theta_0)~~~~\mbox{ in probability }
\end{equation}
as $\vv\rightarrow0$ and $n\rightarrow\8.$ By following the exact
line of \cite[Theorem 2.2]{LSS}, we can deduce that $D_{n,\vv}$ is
invertible on the set
\begin{equation*}
\Gamma_{n,\vv}:=\Big\{\sup_{\theta\in
B_{\eta_{n,\vv}}(\theta_0)}\|(\nn_\theta^{(2)}\Phi_{n,\vv})(\theta)-K_0(\theta_0)\|\le\ff{\aa}{2},~~\hat\theta_{n,\vv}\in
B_{\eta_{n,\vv}}(\theta_0) \Big\}
\end{equation*}
for some constant $\aa>0.$  Let
\begin{equation*}
\mathscr{D}_{n,\vv}=\{D_{n,\vv} \mbox{ is invertible },
\hat\theta_{n,\vv}\in B_{\eta_{n,\vv}}(\theta_0) \}.
\end{equation*}
By virtue of Lemma \ref{le3}, one has
\begin{equation}\label{n1}
\lim_{\vv\rightarrow0,n\rightarrow\8}\P\Big(\sup_{\theta\in
B_{\eta_{n,\vv}}(\theta_0)}\|(\nn_\theta^{(2)}\Phi_{n,\vv})(\theta)-K_0(\theta_0)\|\le\ff{\aa}{2}\Big)=1.
\end{equation}
On the other hand, recall that
\begin{equation}\label{n2}
\lim_{\vv\rightarrow0,n\rightarrow\8}\P\Big(\hat\theta_{n,\vv}\in
B_{\eta_{n,\vv}}(\theta_0)\Big)=1.
\end{equation}
By the fundamental fact: for any events $A,B$,
$\P(AB)=\P(A)+\P(B)-\P(A\cup B)$, we observe that
\begin{equation}\label{n3}
\begin{split}
1\ge\P(\Gamma_{n,\vv})&\ge\P\Big(\sup_{\theta\in
B_{\eta_{n,\vv}}(\theta_0)}\|(\nn_\theta^{(2)}\Phi_{n,\vv})(\theta)-K_0(\theta_0)\|\le\ff{\aa}{2}\Big)\\
&\quad+\P\Big(\hat\theta_{n,\vv}\in
B_{\eta_{n,\vv}}(\theta_0)\Big)-1.
\end{split}
\end{equation}
Thus, taking advantage of \eqref{n1}, \eqref{n2} as well as
\eqref{n3}, we deduce from Sandwich theorem that
\begin{equation}\label{n4}
\P(\mathscr{D}_{n,\vv})\ge \P(\Gamma_{n,\vv})\rightarrow1
\end{equation}
 as $\vv\rightarrow0$ and $n\rightarrow\8$. Set
\begin{equation*}
U_{n,\vv}:=D_{n,\vv}{\bf1}_{\mathscr{D}_{n,\vv}}+I_{p\times
p}{\bf1}_{\mathscr{D}_{n,\vv}^c},
\end{equation*}
where $I_{p\times p}$ is a $p\times p$ identity matrix. For
$S_{n,\vv}:=\vv^{-1}(\hat\theta_{n,\vv}-\theta_0)$, we deduce from
\eqref{a4}  that
\begin{align*}
S_{n,\vv}&=S_{n,\vv}{\bf1}_{\mathscr{D}_{n,\vv}}+ S_{n,\vv}{\bf1}_{\mathscr{D}_{n,\vv}^c}\\
&=U_{n,\vv}^{-1}D_{n,\vv}S_{n,\vv}{\bf1}_{\mathscr{D}_{n,\vv}}+ S_{n,\vv}{\bf1}_{\mathscr{D}_{n,\vv}^c}\\
&=\vv^{-1}U_{n,\vv}^{-1}\{(\nn_\theta\Phi_{n,\vv})(\hat\theta_{n,\vv})-(\nn_\theta\Phi_{n,\vv})(\theta_0)\}{\bf1}_{\mathscr{D}_{n,\vv}}
+ S_{n,\vv}{\bf1}_{\mathscr{D}_{n,\vv}^c}\\
&=-\vv^{-1}U_{n,\vv}^{-1}(\nn_\theta\Phi_{n,\vv})(\theta_0){\bf1}_{\mathscr{D}_{n,\vv}}
+ S_{n,\vv}{\bf1}_{\mathscr{D}_{n,\vv}^c}\\
&\rightarrow I^{-1}(\theta_0)\int_0^T\Upsilon(X_s^0,\theta_0)\d
B(s),
\end{align*}
as $\vv\rightarrow0$ and $n\rightarrow\infty$, where in the forth
identity we dropped the term
$(\nn_\theta\Phi_{n,\vv})(\hat\theta_{n,\vv})$ according to the
notion of LSE and Fermat's lemma, and the last display follows from
Lemma \ref{le2}, \eqref{a0} as well as \eqref{n4} and by noting
$K_0(\theta_0)=I(\theta_0)$. We therefore complete the proof.
\end{proof}

\section{Proof of Example \ref{exa}}

\begin{proof}[Proof of Example \ref{exa}]
It is sufficient to check all of assumptions in Theorems \ref{th1}
and Theorem \ref{th2} are fulfilled.

For any $\zeta,\zeta'\in\C$, $\mu\in\mathcal {P}_2(\C)$ and
$\theta=(\theta^{(1)},\theta^{(2)})^*\in\Theta_0$, set
\begin{equation}\label{ex1}
b_0(\zeta,\zeta'):=-\zeta^3(0)+\zeta(0)+\int_{-r_0}^0\zeta(v)\d
v+\int_{-r_0}^0\zeta'(v)\d v,
\end{equation}
\begin{equation*}
b(\zeta,\mu,\theta):=\theta^{(1)}+\theta^{(2)}\int_\C
b_0(\zeta,\zeta')\mu(\d\zeta') ~~\mbox{ and
}~~\si(\zeta):=\si(\zeta,\mu):=1+\int_{-r_0}^0|\zeta(v)|\d v.
\end{equation*}
Then, \eqref{d1} can be reformulated as \eqref{eq1}. By  \eqref{ex1}
and H\"older's inequality, we find out some constants $c_1,c_2>0$
such that
\begin{equation}\label{d3}
\begin{split}
&\<\zeta_1(0)-\zeta_2(0),b(\zeta_1,\mu,\theta)-b(\zeta_2,\mu,\theta)\>\\
&= \theta^{(2)} \int_\C \<\zeta_1(0)-\zeta_2(0),b_0(\zeta_1,\zeta)-
b_0(\zeta_2,\zeta)\>\mu_t(\d\zeta) \\
&\le c_1\Big\{|\zeta_1(0)-\zeta_2(0)|^2+\int_{-r_0}^0|\zeta_1(v)-\zeta_2(v)|^2\d v\Big\}\\
&\le c_2 \,\|\zeta_1-\zeta_2\|_\8^2,~~~~~\mu\in\mathcal {P}_2(\C),
~~\zeta_1,\zeta_2 \in\C
\end{split}
\end{equation}
Next, we deduce from \eqref{ex1} that for some constant $c_3>0,$
\begin{equation*}
\begin{split}
|b(\zeta,\mu,\theta)-b(\zeta,\nu,\theta)|&\le\theta^{(2)}\Big|\int_\C
b_0(\zeta,\zeta_1)\mu(\d\zeta_1)-\int_\C
b_0(\zeta,\zeta_2)\nu(\d\zeta_2)\Big|\\
&\le\theta^{(2)}\int_\C\int_\C| b_0(\zeta,\zeta_1)-
b_0(\zeta,\zeta_2)|\pi(\d\zeta_1,\d\zeta_2)\\
&\le c_3\,\mathbb{W}_2(\mu,\nu),~~~~\zeta\in\C,~~~\mu,\nu\in\mathcal
{P}_2(\C),
\end{split}
\end{equation*}
in which $\pi\in\mathcal {C}(\mu,\nu)$. Therefore, ({\bf A1}) holds
true. Next,
 for any
$\zeta_1,\zeta_2\in\C$ and $\mu,\nu\in\mathcal {P}_2(\C)$, we obtain
that
\begin{equation*}
 |\si(\zeta_1,\mu)-\si(\zeta_2,\nu) |\le\int_{-r_0}^0|\zeta_1(\theta)-\zeta_2(\theta)|\d\theta\le r_0\|\zeta_1-\zeta_2\|_\8.
\end{equation*}
So ({\bf A2}) is satisfied. For any
$\zeta_1,\zeta_2,\zeta^{(1)},\zeta^{(2)}\in\C$, note that
\begin{equation}\label{ex2}
\begin{split}
&|b_0(\zeta_1,\zeta^{(1)})-b_0(\zeta_2,\zeta^{(2)})|\\&\le
|\zeta_1^3(0)-\zeta_2^3(0)|+|\zeta_1(0)-\zeta_2(0)|+\int_{-r_0}^0|\zeta_1(v)-\zeta_2(v)|\d
v+\int_{-r_0}^0|\zeta^{(1)}(v)-\zeta^{(2)}(v)|\d v\\
&\le
c_4(1+\zeta_1^2(0)+\zeta_2^2(0))|\zeta_1(0)-\zeta_2(0)|+r_0\|\zeta_1-\zeta_2\|_\8+r_0\|
\zeta^{(1)}-\zeta^{(2)}\|_\8\\
&\le
c_5(1+\|\zeta_1\|_\8^2+\|\zeta_2\|_\8^2)\|\zeta_1-\zeta_2\|_\8+r_0\|
\zeta^{(1)}-\zeta^{(2)} \|_\8
\end{split}
\end{equation}
for some constants $c_4,c_5>0.$ Next, we have
\begin{equation}\label{d4}
(\nn_\theta b)(\zeta,\mu,\theta)=\Big(1,\int_\C
b_0(\zeta,\zeta')\mu(\d\zeta')\Big)^*~~~\mbox{ and
}~~~(\nn_\theta(\nn_\theta b))(\zeta,\mu,\theta)={\bf 0}_{2\times2},
\end{equation}
where ${\bf 0}_{2\times2}$ stands for the $2\times 2$-zero matrix.
Thus, \eqref{ex2} and \eqref{d4} enable us to deduce that ({\bf B2})
and ({\bf C}) hold, respectively. Furthermore, due to \eqref{ex2},
we find that
\begin{equation*}
\begin{split}
|b(\zeta_1,\mu,\theta)-b(\zeta_2,\nu,\theta)|&\le\theta^{(2)}\Big|\int_\C
b_0(\zeta_1,\zeta^{(1)})\mu(\d\zeta^{(1)})-\int_\C
b_0(\zeta_2,\zeta^{(2)})\nu(\d\zeta^{(2)})\Big|\\
&\le\theta^{(2)}\int_\C\int_\C| b_0(\zeta_1,\zeta^{(1)})-\int_\C
b_0(\zeta_2,\zeta^{(2)})|\pi(\d\zeta^{(1)},\d\zeta^{(2)})\\
&\le
c_6(1+\|\zeta_1\|_\8^2+\|\zeta_2\|_\8^2)\|\zeta_1-\zeta_2\|_\8+c_6\mathbb{W}_2(\mu,\nu).
\end{split}
\end{equation*}
Therefore, we infer that ({\bf B1})  holds. Next, observe that
\begin{equation*}
|\si^{-2}(\zeta_1,\mu)-\si^{-2}(\zeta_2,\nu)| \le
c_7\|\zeta_1-\zeta_2\|_\8
\end{equation*}
for some $c_7>0.$ Consequently, ({\bf B3}) is true.

The discrete-time EM scheme associated with \eqref{d1} is given by
\begin{equation}
Y^\vv(t_k)=Y^\vv(t_{k-1})+\Big(\theta^{(1)}+\theta^{(2)}\int_\C
b_0(\hat Y^\vv_{t_{k-1}},\zeta)\mathscr{L}_{\hat
Y^\vv_{t_{k-1}}}(\d\zeta)\Big)\dd+\vv\,\si(\hat
Y^\vv_{t_{k-1}})\triangle B_k,~~~k\ge1,
\end{equation}
with $Y^\vv(t)=X^\vv(t)=\xi(t), t\in[-r_0,0],$ where $(\hat
Y^\vv_{t_k})$ is defined as in \eqref{w2}. According to \eqref{eq2},
the contrast function admits the form below
\begin{eqnarray*}
\Psi_{n,\vv}(\theta)&=&\vv^{-2}\delta^{-1}\sum_{k=1}^n\ff{1}{(1+|Y^\vv(t_{k-1})|)^2}\Big|Y^\vv(t_k)
-Y^\vv(t_{k-1})\\
&&\qquad\qquad-\Big(\theta^{(1)}+\theta^{(2)}\int_\C
b_0(\hat Y^\vv_{t_{k-1}},\zeta)\mathscr{L}_{\hat
Y^\vv_{t_{k-1}}}(\d\zeta)\Big)\dd\Big|^2.
\end{eqnarray*}
Observe that
\begin{equation*}
\begin{split}
\ff{\partial}{\partial\theta^{(1)}}\Psi_{n,\vv}(\theta)&=-2\,\vv^{-2}\sum_{k=1}^n\ff{1}{(1+|Y^\vv(t_{k-1})|)^2}\Big\{Y^\vv(t_k)-Y^\vv(t_{k-1})\\
&\quad-\Big(\theta^{(1)}+\theta^{(2)}\int_\C b_0(\hat
Y^\vv_{t_{k-1}},\zeta)\mathscr{L}_{\hat
Y^\vv_{t_{k-1}}}(\d\zeta)\Big)\dd\Big\},
\end{split}
\end{equation*}
and
\begin{equation*}
\begin{split}
\ff{\partial}{\partial\theta^{(2)}}\Psi_{n,\vv}(\theta)&=-2\,\vv^{-2}\sum_{k=1}^n\ff{1}{(1+|Y^\vv(t_{k-1})|)^2}\Big\{Y^\vv(t_k)-Y^\vv(t_{k-1})\\
&\quad-\Big(\theta^{(1)}+\theta^{(2)}\int_\C b_0(\hat
Y^\vv_{t_{k-1}},\zeta)\mathscr{L}_{\hat
Y^\vv_{t_{k-1}}}(\d\zeta)\Big)\dd\Big\} \int_\C b_0(\hat
Y^\vv_{t_{k-1}},\zeta)\mathscr{L}_{\hat Y^\vv_{t_{k-1}}}(\d\zeta).
\end{split}
\end{equation*}
Subsequently, solving the equation below
\begin{equation*}
\ff{\partial}{\partial\theta^{(1)}}\Psi_{n,\vv}(\theta)=\ff{\partial}{\partial\theta^{(2)}}\Psi_{n,\vv}(\theta)=0,
\end{equation*}
we then obtain the LSE
$\hat\theta_{n,\vv}=(\hat\theta_{n,\vv}^{(1)},\hat\theta_{n,\vv}^{(2)})^*$
of the unknown parameter
$\theta=(\theta^{(1)},\theta^{(2)})^*\in\Theta_0$ with the following 
\begin{equation*}
\hat\theta_{n,\vv}^{(1)}=\ff{A_2A_5-A_3A_4}{\dd(A_1A_5-A_4^2)}~~~~~\mbox{
and }~~~~~
\hat\theta_{n,\vv}^{(2)}=\ff{A_1A_3-A_2A_4}{\dd(A_1A_5-A_4^2)},
\end{equation*}
where
\begin{equation*}
A_1:=\sum_{k=1}^n\ff{1}{(1+|Y^\vv(t_{k-1})|)^2},~~~~~~~~~~~~~~~~~~~~~~~~~~~~~~
~~A_2:=\sum_{k=1}^n\ff{Y^\vv(t_k)-Y^\vv(t_{k-1})}{(1+|Y^\vv(t_{k-1})|)^2},
\end{equation*}
\begin{equation*}
A_3:=\sum_{k=1}^n\ff{(Y^\vv(t_k)-Y^\vv(t_{k-1}))\int_\C b_0(\hat
Y^\vv_{t_{k-1}},\zeta)\mathscr{L}_{\hat
Y^\vv_{t_{k-1}}}(\d\zeta)}{(1+|Y^\vv(t_{k-1})|)^2},~~~A_4:=\sum_{k=1}^n\ff{\int_\C
b_0(\hat Y^\vv_{t_{k-1}},\zeta)\mathscr{L}_{\hat
Y^\vv_{t_{k-1}}}(\d\zeta)}{(1+|Y^\vv(t_{k-1})|)^2},
\end{equation*}
and
\begin{equation*}
A_5:=\sum_{k=1}^n\ff{\Big(\int_\C b_0(\hat
Y^\vv_{t_{k-1}},\zeta)\mathscr{L}_{\hat
Y^\vv_{t_{k-1}}}(\d\zeta)\Big)^2}{(1+|Y^\vv(t_{k-1})|)^2}.
\end{equation*}
In terms of Theorem \ref{th1}, $\hat\theta_{n,\vv}\rightarrow\theta$
in probability as $\vv\rightarrow0$ and $n\rightarrow\infty$. Next,
from \eqref{d4}, it follows that
\begin{equation*}
I(\theta_0)=\left(\begin{array}{ccc}
  \int_0^T\ff{1}{(1+|X_s^0|)^2}\d s & \int_0^T\ff{b_0(X_s^0,X_s^0)}{(1+|X_s^0|)^2}\d s\\
  \int_0^T\ff{b_0(X_s^0,X_s^0)}{(1+|X_s^0|)^2}\d s & \int_0^T\ff{b_0^2(X_s^0,X_s^0)}{(1+|X_s^0|)^2}\d s\\
  \end{array}
  \right),
\end{equation*}
and, for $\zeta\in\C,$
\begin{equation*}
\int_0^T\Upsilon(X_s^0,\theta_0)\d B(s)=\left(\begin{array}{c}
  \int_0^T\ff{1}{1+|X^0(s)|}\d B(s) \\
 \int_0^T\ff{ b_0(X_s^0,X_s^0)}{1+|X^0(s)|}\d B(s)\\
  \end{array}
  \right).
\end{equation*}
At last, according to Theorem \ref{th2}, we conclude that
\begin{equation*}
\vv^{-1}(\hat\theta_{n,\vv}-\theta_0)\rightarrow
I^{-1}(\theta_0)\int_0^T\Upsilon(X_s^0,\theta_0)\d B(s)~~~~\mbox{ in
probability }
\end{equation*}
as $\vv\rightarrow0$ and $n\rightarrow\8$ provided that $I(\cdot)$
is positive definite.

\end{proof}

\end{document}